\documentclass[journal,twoside,web]{ieeecolor}
\usepackage{generic}
\usepackage{cite}
\usepackage{amsmath,amssymb,amsfonts}
\usepackage{algorithmic}
\usepackage{graphicx}
\usepackage{algorithm,algorithmic}
\usepackage{textcomp}
\def\BibTeX{{\rm B\kern-.05em{\sc i\kern-.025em b}\kern-.08em
    T\kern-.1667em\lower.7ex\hbox{E}\kern-.125emX}}
\usepackage{mathtools}
\usepackage{stfloats}

\DeclareMathOperator*{\argmin}{arg\,min}

\newtheorem{theorem}{Theorem}[section] 
\newtheorem{lemma}[theorem]{Lemma} 
\newtheorem{assumption}[theorem]{Assumption}
\newtheorem{definition}[theorem]{Definition}
\newtheorem{proposition}[theorem]{Proposition}
\newtheorem{corollary}[theorem]{Corollary}

\newtheorem{aproposition}{Proposition}[section]

\newcommand{\inX}[2]{\in\mathbb{#1}^{#2}}                               
\newcommand{\IqZ}{\begin{bmatrix}I_q\\Z\end{bmatrix}}                   
\newcommand{\sqmat}[4]{\begin{bmatrix}#1&#2\\#3&#4\end{bmatrix}}        
\newcommand{\smallsqmat}[4]{\left[\begin{smallmatrix}#1&#2\\#3&#4\end{smallmatrix}\right]}        
\newcommand{\matvec}[2]{\begin{bmatrix}#1\\#2\end{bmatrix}}             

\newcommand{\Zqr}{\mathcal{Z}_{q,r}}                                      
\newcommand{\im}{\mathrm{im}\,}  
\newcommand{\Zab}[2]{\mathcal{Z}_{#1,#2}}                             
\newcommand{\hDel}{\hat{\Delta}}                                        
\newcommand{\hPhi}{\hat{\Phi}}                                          %
\newcommand{\st}{\text{ s.t. }}                                     
\newcommand{\blkdiag}{{\rm diag}} 
\newcommand{\Hc}[1]{{\mathcal{H}_#1}}
\newcommand{\tr}{\mathrm{trace}\,}
\newcommand{\Ayl}{A_{Y,L}}
\newcommand{\Cyl}{C_{Y,L}}
\newcommand{\univec}[2]{{e_{#1,#2}}}

\newcommand{\SigmaR}{\Sigma_{\mathcal{R}}}
\newcommand{\SigmaN}{\Sigma_{\mathcal{N}}}

\mathtoolsset{showonlyrefs}
\renewcommand{\QED}{\QEDopen}

\usepackage[hidelinks]{hyperref} 
\begin{document}
\title{Data Informativity under Data Perturbation}
\author{Taira Kaminaga, and Hampei Sasahara, \IEEEmembership{Member, IEEE}
\thanks{\copyright\ 2026 IEEE.  Personal use of this material is permitted.
Permission from IEEE must be obtained for all other uses, in any current or future media,
including reprinting/republishing this material for advertising or promotional purposes,
creating new collective works, for resale or redistribution to servers or lists,
or reuse of any copyrighted component of this work in other works.}
\thanks{This is the author-accepted manuscript of:
T. Kaminaga and H. Sasahara,
"Data Informativity under Data Perturbation,"
\emph{IEEE Transactions on Automatic Control},
early access, September 28, 2026,
doi: \href{https://doi.org/10.1109/TAC.2026.3738318}{10.1109/TAC.2026.3738318}.}
\thanks{This work was supported by JSPS KAKENHI Grant Number 24K17296 and JST SPRING Grant Number JPMJSP2108.}
\thanks{T. Kaminaga and H. Sasahara are with Department of Information Physics and Computing, Graduate School of Information Science and Technology, the University of Tokyo, Tokyo, Japan \tt{tairakaminaga@g.ecc.u-tokyo.ac.jp}, \tt{hsasahara@g.ecc.u-tokyo.ac.jp}.}
}

\maketitle

\begin{abstract}
Data informativity provides a theoretical foundation for determining whether collected data are sufficiently informative to achieve specific control objectives in data-driven control frameworks.
In this study, we investigate data informativity under data perturbation, a generalized noise model additive to the whole input-state data matrices constrained within a linear subspace and characterized by a quadratic matrix inequality (QMI).
We derive necessary and sufficient conditions formulated as tractable linear matrix inequalities for data informativity under data perturbation with respect to stabilization and performance guarantees via state feedback, as well as stabilization via output feedback.
Our proposed framework encompasses and extends existing analyses that consider exogenous disturbances and measurement noise, while also relaxing several restrictive assumptions commonly made in prior work.
A central challenge in the data perturbation setting arises from the non-convexity of the set of systems consistent with the data, which renders standard matrix S-procedure techniques inapplicable.
To resolve this issue, we develop a novel matrix S-procedure that does not rely on convexity of the system set.
Furthermore, 
we introduce structured data perturbation as a noise model to address realistic noise settings and derive sufficient conditions for data informativity under this model.
The proposed results are broadly applicable to a wide class of noise models and subsume several existing methodologies as special cases.
\end{abstract}


\section{Introduction}
\label{sec:introduction}

There are two principal approaches to designing controllers using data collected from an unknown dynamical system. The first involves identifying a mathematical model from the data, followed by controller design based on the identified model.
The second, known as \emph{data-driven control,} bypasses explicit model identification and directly synthesizes controllers from data.
The latter approach has garnered significant interest in recent years due to its potential for simpler and more streamlined formulations~\cite{DDC:Hou2013}.
In the data-driven control paradigm, it is essential that the available data contain sufficient information about the control objective of interest.
In this context, the concept of \emph{data informativity} has been introduced as a theoretical framework for characterizing the informational adequacy of data~\cite{DDC:Waarde2020_TAC_Dinfo,DDC:Waarde2023_Cont_sys_mag_informativity}.
Various necessary and sufficient conditions have been developed to assess data informativity across different classes of systems and control objectives.

Since the true system dynamics is unknown, the study of data informativity focuses on determining whether all systems consistent with the observed data can be controlled to meet a specified objective.
This is formally expressed through an inclusion relationship between two sets: the set of systems consistent with the data must be contained within the set of systems for which a particular controller achieves the desired control objective.
In this context, a common assumption is that the noise sequence satisfies a quadratic matrix inequality (QMI), which constrains the noise via a matrix-valued quadratic function~\cite{DDCQMI:Waarde2022_TAC_origin}.
The inclusion condition between the system sets can then be reformulated as a linear matrix inequality (LMI) through the application of the matrix S-procedure~\cite{DDCQMI:Waarde2022_TAC_origin,DDCQMI:Waarde2023_siam_qmi}.
This resulting LMI not only serves as a criterion for assessing data informativity but also plays a constructive role in controller synthesis.

The foundational framework described above has been introduced in~\cite{DDCQMI:Waarde2022_TAC_origin}, where the authors examined data informativity for achieving $\mathcal{H}_2$ and $\mathcal{H}_\infty$ performance guarantees.
Their analysis has considered systems subject to exogenous disturbances, defined as additive perturbations to the subsequent internal state, constrained by a QMI.
Under similar noise settings, this QMI-based framework has been extended to a broader range of control objectives~\cite{DDCQMI:Steentjes2022_Cont_Sys_Let_Covariance,DDCQMI:Waarde2024_TAC_AR, DDCQMI:Waarde2024ECC_disp, DDCQMI:Burohman2023_TAC_reducedorder}.
Concurrently, the matrix S-lemma, which serves as the central analytical tool, has been generalized by relaxing its underlying assumptions~\cite{DDCQMI:Waarde2023_siam_qmi}.
In parallel, several studies have considered the impact of noises during data acquisition processes, which we collectively term measurement noise.
The studies~\cite{DDCQMI:BISOFFI2024_CSL,DDCQMI:Lidong2026_automatica} have addressed stabilization under measurement noise constrained by an energy-bound type QMI, and the study~\cite{DDCQMI:Hu2025_RDPC} has addressed predictive control under both exogenous disturbances and measurement noise.
While the majority of the aforementioned works assume a single QMI constraint imposed on the entire noise sequence, the studies~\cite{DDCQMI:BISOFFI2024_CSL,DDCQMI:BISOFFI2021,DDCQMI:Hu2022_CDC} investigate an alternative formulation with multiple QMIs involving instantaneous bounds, where the noise at each time step independently satisfies a QMI constraint.

While various analyses under different noise models have been conducted, existing approaches rely on distinct assumptions and analytical methods. 
In particular, studies addressing measurement noise~\cite{DDCQMI:BISOFFI2024_CSL, DDCQMI:Lidong2026_automatica} commonly assume that the data have a sufficiently large signal-to-noise ratio (SNR). 
However, given that the fundamental purpose of data informativity is to accurately assess the informational adequacy of data, imposing assumptions on the data itself is undesirable.
Relaxing these assumptions and developing a unified framework that encompasses different noise models remains an open challenge.

In this study, we propose \emph{data perturbation}
as a generalized noise model, unifying the existing approaches for exogenous disturbances and measurement noise.
We define it as additive noise affecting the input, state, and subsequent state or output data matrices subject to linear constraints.
Within this framework,
we first analyze data informativity for stabilization and for achieving $\mathcal{H}_2$ and $\mathcal{H}_\infty$ performance guarantees via state feedback.
We derive LMI conditions equivalent to data informativity under a structural condition on the perturbation.
Our results generalize previous analyses involving exogenous disturbances~\cite{DDCQMI:Waarde2022_TAC_origin,DDCQMI:BISOFFI2022_Petersen,DDCQMI:Waarde2023_siam_qmi,DDCQMI:Steentjes2022_Cont_Sys_Let_Covariance} and extend those addressing measurement noise~\cite{DDCQMI:BISOFFI2024_CSL} by generalizing the energy-bound formulation to a broader class of QMI constraints, while also removing the SNR assumption.
Technically, a primary challenge is that the set of systems consistent with the data lacks well-behaved properties within our generalized setting.
Although we successfully characterize this set using the solution set of a QMI, we reveal that this characterization can be conservative, unlike in prior studies. 
To resolve this issue, we derive a certain structural condition on the perturbation under which the characterization becomes non-conservative.
Furthermore, we confirm that the standard cases involving exogenous disturbances and measurement noise satisfy this condition. 
A further complication arises because this set of systems does not satisfy geometric conditions such as convexity, precluding the use of the existing matrix S-procedure.
We overcome this issue by developing a novel matrix S-procedure that does not require convexity of the system set.
This extension broadens the applicability of data informativity analysis to more general noise models.

Second, we establish conditions for data informativity with respect to stabilization via output feedback under the proposed data perturbation model.
Notably, our results do not rely on the full-rank condition on the data, which was assumed in prior analyses~\cite{DDCQMI:Steentjes2022_Cont_Sys_Let_Covariance,DDCQMI:Waarde2024_TAC_AR}, nor do they require the assumption of a sufficiently large SNR, as imposed in studies dealing with measurement noise~\cite{DDCQMI:Lidong2026_automatica}.
An additional technical challenge in the output feedback setting is the mismatch in the variable domains of the QMI formulations.
Due to the difference in variables, the standard matrix S-procedure is not directly applicable.
To address this issue, we develop a transformation that reformulates the QMI involving state-space matrices into an equivalent QMI involving coefficient matrices.
This transformation departs from the approaches taken in earlier works~\cite{DDCQMI:Steentjes2022_Cont_Sys_Let_Covariance,DDCQMI:Waarde2024_TAC_AR} and enables us to remove the full-rank assumption on the data.
As a result, our framework offers a more general and flexible analysis of output feedback stabilization.

Finally, we propose \emph{structured data perturbation} as a more generalized noise representation, defined as a sum of independent data perturbations. 
This representation can capture various realistic noise models, including the superposition of exogenous disturbances and measurement noise, Hankel-structured perturbations, and instantaneously or element-wise bounded perturbations.
Within this framework, we derive sufficient conditions for data informativity in the form of LMIs.
To assess data informativity in this setting, we consider an outer QMI approximation of the combined noise region.
First, we derive an LMI condition that ensures the QMI solution set sufficiently covers the entire region affected by the structured data perturbation.
Second, using our earlier results for data perturbations constrained by a single QMI, we establish an additional LMI condition that is sufficient for guaranteeing data informativity.
We develop a co-design strategy in which the LMI for noise region approximation and the LMI for verifying data informativity are solved simultaneously.
Although this co-design leads to a computationally intractable bilinear structure, we show that this bilinearity can be eliminated without loss of generality.
Our results encompass the analysis of instantaneous noise bounds in~\cite{DDCQMI:BISOFFI2021} as a special case, while also enabling us to address other realistic noise models.

The main contributions of this study are as follows:
(1) We introduce the notion of data perturbation, a unified noise framework that encompasses both exogenous disturbances and measurement noise.
This framework enables the integration of data informativity analyses previously considered separately into a single, cohesive analysis.
(2) Within the data informativity framework, we derive LMI conditions that are equivalent to data informativity for stabilization and performance guarantees via state feedback, as well as stabilization via output feedback, under a structural condition on the perturbation.
(3) In the above analysis, our technical contribution lies primarily in the proposed novel matrix S-procedure. 
It does not require geometric preconditions on the set of systems consistent with the data, making it highly suitable for the data perturbation setting.
(4)
We propose structured data perturbation as a noise representation to capture more realistic settings. Furthermore, we show how to extend the informativity analysis under data perturbation to the case of structured data perturbation.
(5) 
We provide several illustrative numerical examples to demonstrate the applicability of our results.
These examples include: (i) a visualization with a one-dimensional system,
(ii) controller synthesis without the large SNR assumption in the measurement noise analysis,
(iii) control performance evaluation with respect to variations in data length, and
(iv) verification of the effectiveness of the co-design method for structured data perturbation.

A preliminary version of this work has been presented in~\cite{DDCQMI:Kaminaga2025_ACC}, but it only discusses data informativity for stabilization using state-feedback control.
In addition, this article provides rigorous proofs of the theoretical claims that are not included in the previous version.

This paper is organized as follows. 
Sec.~\ref{sec:preliminary} reviews basic properties of QMI.
Sec.~\ref{sec:quadratic_stabilization} introduces data perturbation and provides an equivalent LMI condition to data informativity for quadratic stabilization.
Sec.~\ref{sec:optimal} and Sec.~\ref{sec:inout} extend the result in Sec.~\ref{sec:quadratic_stabilization} to optimal control and output feedback, respectively.
Sec.~\ref{sec:structured} introduces structured data perturbation framework and derives a sufficient condition for data informativity.
Sec.~\ref{sec:exam} provides several numerical experiments, and finally, Sec.~\ref{sec:conclusion} concludes this paper.

\subsection*{Notation}
The symbol $\mathbb{S}^n$ denotes the set of $n\times n$ symmetric matrices, $I_n$ denotes the $n\times n$ identity matrix, and $0_n$ and $0_{m,n}$ denote the $n\times n$ and $m\times n$ zero matrices, respectively (the subscript is omitted when the dimension is clear from the context).
For a matrix $M$, $M^\top$ denotes its transpose, $\tr M$ denotes its trace, and $M^\dagger$ denotes its pseudo-inverse.
For a symmetric matrix $M$, $M\succ (\succeq) 0$ and $M\prec (\preceq) 0$ denote its positive and negative (semi)definiteness, respectively.
The notation $M^\frac{1}{2}$ denotes the positive semidefinite square root of $M\succeq 0$, and $\blkdiag(M_1,M_2,\dots)$ denotes the block diagonal matrix with blocks $M_1,M_2,\dots$.
The expression $M|D\coloneqq A-BD^\dagger C$ denotes the generalized Schur complement of $D$ in $M=\smallsqmat{A}{B}{C}{D}$.
Finally, $e_{k,n}$ denotes the $n$-dimensional vector with the $k$-th element being $1$ and others $0$.

\section{Preliminary}\label{sec:preliminary}
We summarize important properties related to QMIs.
A more detailed discussion can be found in~\cite{DDCQMI:Waarde2023_siam_qmi}.
A QMI of the matrix $Z\inX{R}{r\times q}$ is defined as
\begin{equation}\label{qmi:dfn}
   \IqZ^\top  \Pi\IqZ\succeq 0
\end{equation}
with $\Pi\inX{S}{q+r}$.
The set $\Zqr(\Pi)\coloneqq\{Z\inX{R}{r\times q}\vert{}\eqref{qmi:dfn}\}$ denotes the solution set of \eqref{qmi:dfn}, and similarly, $\Zqr^+(\Pi)$ denotes the solution set of a strict QMI.
To analyze these solution sets, we partition $\Pi$ as $\Pi=\smallsqmat{\Pi_{11}}{\Pi_{12}}{\Pi_{21} }{\Pi_{22}}$, where $\Pi_{11}\inX{S}{q}$ and $\Pi_{22}\inX{S}{r}$.

First, we introduce the term \emph{matrix ellipsoid}, which characterizes geometric properties of QMIs.
We say a set $\Zqr(\Pi)$ to be a matrix ellipsoid 
when the matrix $\Pi\inX{S}{q+r}$ in the QMI \eqref{qmi:dfn} satisfies
\begin{equation}\label{qmi:N_condition}
   \Pi_{22}\preceq 0,\quad \ker \Pi_{22}\subseteq\ker \Pi_{12},\quad \Pi|\Pi_{22}\succeq 0.
\end{equation}
Since $\Pi_{22}(I - \Pi_{22}^\dagger\Pi_{22})=0$, the columns of $I - \Pi_{22}^\dagger\Pi_{22}$ belong to $\ker \Pi_{22}$. 
Thus, the condition $\ker \Pi_{22}\subseteq\ker \Pi_{12}$ implies $\Pi_{12}(I-\Pi_{22}^\dagger\Pi_{22})=0$, i.e., $\Pi_{12}\Pi_{22}^\dagger\Pi_{22}=\Pi_{12}$. 
Using this relation, it follows from direct computation that
\begin{equation}\label{N_bunkai}
   \Pi=\sqmat{I_q}{\Pi_{12}\Pi_{22}^\dagger}{0}{I_r}\sqmat{\Pi|\Pi_{22}}{0}{0}{\Pi_{22}}\sqmat{I_q}{\Pi_{12}\Pi_{22}^\dagger}{0}{I_r}^\top.
\end{equation}
Therefore, when \eqref{qmi:N_condition} holds, the QMI \eqref{qmi:dfn} can be written as
{\small\begin{align}
    \hspace{-15pt}\IqZ^\top\! \Pi\IqZ=&\Pi|\Pi_{22}+(Z+\Pi_{22}^\dagger \Pi_{21})^\top \Pi_{22}(Z+\Pi_{22}^\dagger \Pi_{21})\label{qmi:tenkai}\\
    =&Q^2-(Z-Z_c)^\top R^2(Z-Z_c)\succeq 0,\label{qmi:Z_ellip}
\end{align}}\noindent
where $Q\coloneqq(\Pi|\Pi_{22})^\frac{1}{2},R\coloneqq (-\Pi_{22})^\frac{1}{2}$, and $Z_c\coloneqq-\Pi_{22}^\dagger \Pi_{21}$.
The description \eqref{qmi:Z_ellip} is an extension of the standard description of ellipsoids~\cite[Eq. (3.9)]{boyd1994linear},
and implies that the set $\Zqr(\Pi)$ is convex and symmetric with respect to the point $Z_c$.
We define the set of the symmetric matrices satisfying \eqref{qmi:N_condition} as
\begin{equation*}
   \mathbf{\Pi}_{q, r} \coloneqq \left\{\sqmat{\Pi_{11}}{\Pi_{12}}{\Pi_{21} }{\Pi_{22}}\inX{S}{q+r}\middle|\eqref{qmi:N_condition}\right\}.
\end{equation*}

The term matrix ellipsoid has originally been introduced in~\cite[Sec. 2.2]{DDCQMI:BISOFFI2021}, where stricter conditions $\Pi_{22}\prec 0$ and $\Pi|\Pi_{22}\succ 0$ than \eqref{qmi:N_condition} are adopted.
This condition requires that the set $\Zqr(\Pi)$ is bounded and has nonempty interior~\cite[Theorem 3.2]{DDCQMI:Waarde2023_siam_qmi}.
In this paper, however, we use the condition \eqref{qmi:N_condition} introduced in~\cite{DDCQMI:Waarde2023_siam_qmi} as the definition of matrix ellipsoids, based on its connection with the matrix S-lemma~\cite[Corollary 4.13]{DDCQMI:Waarde2023_siam_qmi} given below.
Note that the matrix ellipsoid defined with \eqref{qmi:N_condition} does not necessarily ensure the boundedness and the existence of nonempty interiors.

Next, we introduce the (strict) matrix S-lemma, proposed in~\cite{DDCQMI:Waarde2022_TAC_origin,DDCQMI:Waarde2023_siam_qmi}, as a generalization of standard S-lemma~\cite{MATH:Polik2007}.
\begin{proposition}[{\cite[Corollary 4.13]{DDCQMI:Waarde2023_siam_qmi}}]\label{prop:Slem_beta}
   Let $\Omega, \Pi\inX{S}{q+r}$.
   Assume that $\Pi\in\mathbf{\Pi}_{q, r}$ and $\Omega_{22}\preceq 0$.
   Then, $\Zqr(\Pi)\subseteq\Zqr^+(\Omega)$ holds if and only if there exist scalars $\alpha\geq 0$ and $\beta>0$ such that
     $\Omega-\alpha \Pi\succeq \blkdiag(\beta I_q,0_r)$.
\end{proposition}

In the application to control theory, the matrix $Z$ corresponds to system matrices.
Proposition~\ref{prop:Slem_beta} transforms the inclusion relation between two sets of systems $\Zqr(\Pi)$ and $\Zqr^+(\Omega)$ into a tractable LMI condition.

\section{Data Informativity for Quadratic Stabilization under Data Perturbations}\label{sec:quadratic_stabilization}
In this section, we introduce the notion of \emph{data perturbation}, a generalized framework that encompasses exogenous disturbances and measurement noise treated in existing studies.
Further, we examine the data informativity problem for quadratic stabilization under data perturbation.
We show that the data informativity for quadratic stabilization can be characterized through an LMI.
The results presented in this section serve as a basis for the analyses in the subsequent sections.

\subsection{Data Perturbation}\label{subsec:dptb_dfn}

Consider a discrete-time linear time-invariant system $x_+=A^*x+B^*u$ with state $x\in\mathbb{R}^n$ and input $u\in\mathbb{R}^m$.
We assume that the system matrices $(A^*,B^*)$ are unknown, but instead, (noisy) input and output data are available.
For the system, the tuple of ``clean'' $T$-long offline batch data $(X_+^*,X^*,U^*)$ obeys the dynamics, i.e.,
\begin{equation}\label{eq:cleandata}
 X_{+}^*=A^*X^*+B^*U^*,
\end{equation}
where $X_{+}^*\in\mathbb{R}^{n\times T}, X^*\in\mathbb{R}^{n\times T}, U^*\in\mathbb{R}^{m\times T}$.
In this study, we suppose that measured data $(X_+,X,U)$ is perturbed by additive noise $(\Delta_{X_+},\Delta_X,\Delta_U)$ such that
\begin{equation}\label{eq:measured_data}
 X_+=X_+^*+\Delta_{X_+},\quad X=X^*+\Delta_X,\quad U=U^*+\Delta_U,
\end{equation}
where the dimensions of the matrices are compatible with the clean data.

We define the notion of \emph{data perturbation} as the concatenated perturbation matrix
\begin{equation}
 \Delta\coloneqq [\Delta_{X_+}^\top\ -\Delta_X^\top\ -\Delta_U^\top]^\top.
\end{equation}
For generality, we assume the prior knowledge that data perturbation is constrained within a linear subspace such that
\begin{equation}
 \im\Delta\subseteq\im E,
\end{equation}
where $E\inX{R}{(2n+m)\times \hat{n}}$ is a known matrix.
This constraint is not restrictive; rather, it allows the data perturbation to represent exogenous disturbance and measurement noise considered in existing studies.
Indeed, taking $E=[I_n\ 0_{n,n+m}]^\top$ leads to
\begin{equation}\label{eq:system_noise}
 X_{+}=A^*X+B^*U+\Delta_{X_+},
\end{equation}
which is equivalent to the case under exogenous disturbances~\cite{DDCQMI:Waarde2022_TAC_origin, DDCQMI:Waarde2023_siam_qmi}.
Alternatively, taking $E=I_{2n+m}$ implies that all data are perturbed during the measurement process, which is equivalent to the case under measurement noise~\cite{DDCQMI:BISOFFI2024_CSL}.

Additionally, we assume that the data perturbation is constrained by a QMI.
Building on the introduced structure, we make the following assumption.
\begin{assumption}\label{ass:noiseset}
 Data perturbation satisfies
  \begin{equation}\label{noise_set}
    \Delta\in\mathcal{D}\coloneqq \left\{E\hDel|\hDel^\top\in\Zab{\hat{n}}{T}(\hPhi)\right\}
  \end{equation}
  with known matrices $E\inX{R}{(2n+m)\times \hat{n}}$ and $\hPhi\in\mathbf{\Pi}_{\hat{n},T}$.
\end{assumption}

The matrix $\hPhi$ can be obtained from standard prior knowledge regarding the noise.
For instance, when the energy of $\hDel$ is bounded by a known matrix $\hat{\Theta}\succeq 0$, i.e., $\hDel\hDel^\top\preceq \hat{\Theta}$, we can explicitly construct $\hPhi$ as {$\hPhi=\blkdiag(\hat{\Theta},-I_T)$}.
In this manner, the QMI~$\hDel^\top\in\Zab{\hat{n}}{T}(\hPhi)$ can capture various practical constraints on noise, including energy bounds and bounds on sample covariance matrix~\cite{DDCQMI:Waarde2022_TAC_origin,DDCQMI:Waarde2023_siam_qmi}.

We aim at designing a state-feedback controller $u=Kx$ that guarantees a predefined performance criterion with the perturbed data $(X_+,X,U)$ and the prior knowledge of the data perturbation.

\subsection{Problem Formulation under Data Perturbation}
We formulate the data informativity problem for quadratic stabilization.
From~\eqref{eq:cleandata}, we observe that the measured data and the data perturbation satisfy the relationship $[I_n\ A^*\ B^*]\mathbf{X}=[I_n\ A^*\ B^*]\Delta$ where $\mathbf{X}=[X_+^\top\ -X^\top\ -U^\top]^\top$.
Accordingly, we define the system matrices $(A,B)$ consistent with the data as those satisfying
\begin{equation}\label{stateeq_mat2}
  [I\ A\ B]\mathbf{X}=[I\ A\ B]\Delta
\end{equation}
for some $\Delta\in\mathcal{D}$.
Similarly, the set of systems consistent with the data is defined by
\begin{equation}\label{dfn_Sigma}
  \Sigma\coloneqq\left\{(A, B)\middle| \exists \Delta\in\mathcal{D}\st \eqref{stateeq_mat2} \text{ holds}\right\}.
\end{equation}
Note that $\Sigma$ is nonempty because $(A^*,B^*)\in\Sigma$.

We define data informativity for quadratic stabilization under data perturbation as follows.
\begin{definition}\label{dfn:datainfo_qstab}
  The data $(X_+,X,U)$ is informative for quadratic stabilization under data perturbation if there exist a positive definite $P\inX{S}{n}$ and a controller $K\inX{R}{m\times n}$ such that the Lyapunov inequality
  \begin{equation}\label{qstab}
    P-(A+BK)P(A+BK)^\top\succ 0
  \end{equation}
  holds for any $(A,B)\in\Sigma$.
\end{definition}

The problem here is to derive a computationally tractable condition that characterizes the data informativity for quadratic stabilization under data perturbation.

\emph{Remark:}
We compare our problem setting with those in the existing works~\cite{DDCQMI:Waarde2023_siam_qmi,DDCQMI:BISOFFI2024_CSL}.
First, when $E=[I_n\ 0_{n,n+m}]^\top$, which corresponds to the exogenous disturbance case, our problem formulation is exactly the same as that in~\cite[Sec. 2]{DDCQMI:Waarde2023_siam_qmi}.
Next, when $E=I_{2n+m}$, which corresponds to the measurement noise case, our problem represents a less restrictive scenario than that in~\cite{DDCQMI:BISOFFI2024_CSL}.
In~\cite{DDCQMI:BISOFFI2024_CSL}, the measurement noise $\Delta$ is assumed to satisfy the energy constraint $\Delta\Delta^\top\preceq \Theta$
with a known positive semidefinite matrix $\Theta$.
Since the energy constraint can be written as a QMI with $\hPhi = \blkdiag(\Theta,-I_T)\in\mathbf{\Pi}_{2n+m,T}$, our problem setting encompasses this situation.
Additionally, \cite{DDCQMI:BISOFFI2024_CSL} adopts the assumption
\begin{equation}\label{eiv:assumption}
  \matvec{X}{U}\matvec{X}{U}^\top-\Theta_{22}\succ 0,
\end{equation}
where $\Theta_{22}\inX{S}{n+m}$ is the (2,2)-block submatrix of $\Theta$.
This assumption implies that the signal-to-noise ratio (SNR) is sufficiently large.
However, our analysis will show that this assumption is not necessary.

\subsection{QMI Characterization of Consistent System Set}\label{subsec:qstab_sysset}

We provide an equivalent QMI characterization of the consistent system matrices $(A,B)\in\Sigma$.
Note that, in the exogenous disturbance case, the noise matrix is uniquely determined from~\eqref{eq:system_noise} given system matrices $(A,B)$ and measured data $(X_+,X,U)$.
Leveraging this uniqueness, we can readily transform the QMI of the noise into a corresponding QMI for the consistent system matrices~\cite[Lemma 4]{DDCQMI:Waarde2022_TAC_origin}.
However, in the context of data perturbation, the noise matrix is not uniquely determined from~\eqref{stateeq_mat2}, and hence its derivation is not straightforward.

We introduce the system set $\SigmaR$ defined as the solution set of a QMI:
\begin{equation}
    \SigmaR\coloneqq\{(A,B)|[A\ B]^\top\in\Zab{n}{n+m}(N)\},
\end{equation}
where $N$ is defined as
\begin{equation}\label{dfn_N}
    N \coloneqq [E\ \mathbf{X}]\hPhi[E\ \mathbf{X}]^\top.
\end{equation}
We will show that $\SigmaR$ coincides with $\Sigma$ in most cases.
First, we establish $\Sigma\subseteq\SigmaR$, which follows directly from the definitions of $\Sigma$ and $\mathcal{D}$.
\begin{lemma}\label{lem:conservative_characterization}
    The inclusion $\Sigma\subseteq\SigmaR$ holds.
\end{lemma}
\begin{proof}
    Let $(A,B)\in \Sigma$.
    From the definition of $\Sigma$ and $\mathcal{D}$, there exists $\hDel\inX{R}{\hat{n}\times T}$ 
    such that $[I\ \hDel]\hPhi[I\ \hDel]^\top\succeq 0$ and $[I\ A\ B]\mathbf{X}=[I\ A\ B]E\hDel$.
    Combining them leads to
    $[I\ A\ B][E\ \mathbf{X}]\hPhi[E\ \mathbf{X}]^\top [I\ A\ B]^\top\succeq 0.$
    This inequality is equivalent to $[A\ B]^\top\in\Zab{n}{n+m}(N)$, thus $(A,B)\in \SigmaR$.
\end{proof}

Lemma~\ref{lem:conservative_characterization} provides a conservative characterization of $\Sigma$ via the QMI defined by $N$, which allows us to derive a conservative condition for the quadratic stabilization of $\Sigma$.

Next, we consider conditions for $\Sigma = \SigmaR$.
We first show that $\Sigma$ can be written as the intersection of $\SigmaR$ and another system set, $\SigmaN$, defined as the set of $(A,B)$ that satisfies
\begin{equation}\label{dfn_sigmaN}
    [I\ A\ B](\mathbf{X}-\Delta)(I-\hPhi_{22}\hPhi_{22}^\dagger)=0
\end{equation}
for some $\Delta\in\mathcal{D}$.
This leads to the following lemma.
\begin{lemma}\label{lem:sigma}
    Under Assumption~\ref{ass:noiseset}, $\Sigma = \SigmaR\cap\SigmaN$ holds.
\end{lemma}
\begin{proof}
    To facilitate the proof, we recall the explicit parametrization of $\Zab{\hat{n}}{T}(\hPhi)$ provided in~\cite[Theorem 3.3]{DDCQMI:Waarde2023_siam_qmi}.
    It states that $\hDel^\top\in\Zab{\hat{n}}{T}(\hPhi)$ if and only if there exist matrices $M_1,M_2\inX{R}{\hat{n}\times T}$ satisfying $M_1M_1^\top\preceq I$ and
    \begin{equation}\label{prf_explicit_hDel}
        \hDel = \hDel_c+\hat{Q}M_1\hat{R}^\dagger + M_2(I-\hPhi_{22}\hPhi_{22}^\dagger),
    \end{equation}
    where $\hDel_c\coloneqq-\hPhi_{12}\hPhi_{22}^\dagger$, $\hat{Q}\coloneqq(\hPhi|\hPhi_{22})^{1/2}$, and $\hat{R}\coloneqq(-\hPhi_{22})^{1/2}$.
    For brevity, we define $S\coloneqq[I\ A\ B]$.

    First, we prove that $\Sigma\subseteq\SigmaR\cap\SigmaN$.
    Since $\Sigma\subseteq\SigmaR$ was established in Lemma~\ref{lem:conservative_characterization}, 
    it suffices to prove $\Sigma\subseteq\SigmaN$.
    This holds because \eqref{stateeq_mat2} implies \eqref{dfn_sigmaN}.

    Next, we prove the converse.
    Let $(A,B)\in\SigmaR\cap\SigmaN$.
    From $(A,B)\in\SigmaR$, we have $S[E\ \mathbf{X}]\hPhi[E\ \mathbf{X}]^\top S^\top\succeq 0$.
    Since Assumption~\ref{ass:noiseset} holds, we can decompose $\hPhi$ as in \eqref{N_bunkai}, yielding
    \begin{equation}
        SE\hat{Q}^2E^\top S^\top-S(\mathbf{X}-E\hDel_c)\hat{R}^2(\mathbf{X}-E\hDel_c)^\top S^\top\succeq 0.
    \end{equation}
    Applying \cite[Lemma A.1 (a)]{DDCQMI:Waarde2023_siam_qmi} to this inequality guarantees the existence of $M_1\inX{R}{\hat{n}\times T}$ such that $M_1M_1^\top\preceq I$ and
    \begin{equation}\label{prf_sigmaR}
      S(\mathbf{X}-E\hDel_c)\hat{R}=SE\hat{Q}M_1.
    \end{equation}
    In addition, it follows from $(A,B)\in\SigmaN$ that there exists $\hDel_{\mathcal{N}}$ satisfying $\hDel_{\mathcal{N}}^\top\in\Zab{\hat{n}}{T}(\hPhi)$ and
    \begin{equation}\label{prf_stateeq_N}
        S\mathbf{X}(I-\hPhi_{22}\hPhi_{22}^\dagger)=SE\hDel_{\mathcal{N}}(I-\hPhi_{22}\hPhi_{22}^\dagger).
    \end{equation}
    We then construct a perturbation matrix $\Delta$ as
    \begin{equation}
        \Delta\coloneqq E\hDel,\quad \hDel\coloneqq \hDel_c+\hat{Q}M_1\hat{R}^\dagger+\hDel_{\mathcal{N}}(I-\hPhi_{22}\hPhi_{22}^\dagger).
    \end{equation}
    By the explicit representation~\eqref{prf_explicit_hDel}, it follows that $\hDel^\top\in\Zab{\hat{n}}{T}(\hPhi)$, which implies $\Delta\in\mathcal{D}$.
    Furthermore, $\Delta$ satisfies the equation~\eqref{stateeq_mat2}:
    \begin{subequations}
    \begin{align}
        \hspace{-20pt}S\mathbf{X}&=S\mathbf{X}\hat{R}\hat{R}^\dagger+S\mathbf{X}(I-\hPhi_{22}\hPhi_{22}^\dagger)\label{prf_stateeq_1}\\
        &=SE(\Delta_c\hat{R} +\hat{Q}M_1)\hat{R}^\dagger + SE\hDel_{\mathcal{N}}(I-\hPhi_{22}\hPhi_{22}^\dagger)\label{prf_stateeq_2}\\
        &=SE(\Delta_c +\hat{Q}M_1\hat{R}^\dagger + \hDel_{\mathcal{N}}(I-\hPhi_{22}\hPhi_{22}^\dagger))\label{prf_stateeq_3}\\
        &=S\Delta.\label{prf_stateeq_4}
    \end{align} 
    \end{subequations}
    Here, \eqref{prf_stateeq_1} uses $\hat{R}\hat{R}^\dagger = \hPhi_{22}\hPhi_{22}^\dagger$ (since $\hat{R}=(-\hPhi_{22})^{1/2}$); \eqref{prf_stateeq_2} follows from \eqref{prf_sigmaR} and \eqref{prf_stateeq_N}; \eqref{prf_stateeq_3} holds because $\hDel_c\hat{R}\hat{R}^\dagger =\hDel_c\hPhi_{22}\hPhi_{22}^\dagger=\hDel_c$; and \eqref{prf_stateeq_4} follows from the definition of $\Delta$.
    Therefore, $(A,B)\in\SigmaR\cap\SigmaN$ satisfies \eqref{stateeq_mat2} with $\Delta\in\mathcal{D}$, which means that $(A,B)\in\Sigma$.
\end{proof}

It follows from Lemma~\ref{lem:conservative_characterization} and Lemma~\ref{lem:sigma} that $\Sigma=\SigmaR$ holds if and only if $\SigmaR\subseteq\SigmaN$ holds.
However, this inclusion does not hold in general, and verifying it is difficult.
Therefore, we present a structural condition on the perturbation set $\mathcal{D}$ such that $\SigmaN$ contains the entire space, leading to the equivalence $\Sigma=\SigmaR$
\begin{theorem}\label{thm:sys_set}
    Let Assumption~\ref{ass:noiseset} hold.
    If 
    \begin{equation}\label{perterbation_condition}
        \im E\supseteq\im[I_n\ 0_{n,n+m}]^\top
    \text{ or }\hPhi_{22}\prec 0,
    \end{equation}
    then $\SigmaN=\mathbb{R}^{n\times n}\times\mathbb{R}^{n\times m}$, which implies $\Sigma=\SigmaR$.
\end{theorem}
\begin{proof}
    From the definition of $\SigmaN$ and $\mathcal{D}$, $(A,B)\in \SigmaN$ if and only if there exists $\hDel_\mathcal{N}$ such that $\hDel_\mathcal{N}^\top\in\Zab{\hat{n}}{T}$ and \eqref{prf_stateeq_N} holds with $S=[I\ A\ B]$.
    By substituting the explicit representation~\eqref{prf_explicit_hDel} into \eqref{prf_stateeq_N}, it follows that $(A,B)\in \SigmaN$ if and only if there exists $M_2\inX{R}{\hat{n}\times T}$ such that
    $[I\ A\ B]\mathbf{X}(I-\hPhi_{22}\hPhi_{22}^\dagger)=[I\ A\ B]EM_2(I-\hPhi_{22}\hPhi_{22}^\dagger)$,
    where we use the fact that $\hDel_c(I-\hPhi_{22}\hPhi_{22}^\dagger)=0$ and $\hat{R}^\dagger (I-\hPhi_{22}\hPhi_{22}^\dagger)=0$. These equalities hold because $\ker \hDel_c\supseteq\ker \hPhi_{22}$, and $\ker \hat{R}^\dagger =\ker \hPhi_{22}$.
    Due to the arbitrariness of $M_2$, $(A,B)\in \SigmaN$ is equivalent to
    \begin{equation}\label{prf_SigmaN}
        \im [I\ A\ B]\mathbf{X}(I-\hPhi_{22}\hPhi_{22}^\dagger)\subseteq \im [I\ A\ B]E.
    \end{equation}

    When $\im E\supseteq\im[I_n\ 0_{n,n+m}]^\top$, there exists $F$ such that $EF=[I_n\ 0_{n,n+m}]^\top$.
    Then
    $\im[I\ A\ B]E\supseteq\im [I\ A\ B]EF=\im I = \mathbb{R}^n$
    for any $(A,B)$.
    Thus, \eqref{prf_SigmaN} is satisfied for any $(A,B)$, and $\SigmaN=\mathbb{R}^{n\times n}\times\mathbb{R}^{n\times m}$ holds.
    When $\hPhi_{22}\prec 0$, we have $[I\ A\ B]\mathbf{X}(I-\hPhi_{22}\hPhi_{22}^\dagger)=0$.
    Then, \eqref{prf_SigmaN} is satisfied for any $(A,B)$ in this case as well, and $\SigmaN=\mathbb{R}^{n\times n}\times\mathbb{R}^{n\times m}$.
    Finally, since $\Sigma=\SigmaR\cap\SigmaN$ holds by Lemma~\ref{lem:sigma}, the fact that $\SigmaN$ is the entire space implies $\Sigma = \SigmaR$.
\end{proof}

Theorem~\ref{thm:sys_set} elucidates the fact that previous researches that treat exogenous disturbance~\cite{DDCQMI:Waarde2022_TAC_origin,DDCQMI:Waarde2023_siam_qmi} and measurement noise~\cite{DDCQMI:BISOFFI2024_CSL} can provide non-conservative QMI characterizations of $\Sigma$.
In the exogenous disturbance case with $E=[I_n\ 0_{n,n+m}]^\top$ and the measurement noise case with $E=I_{2n+m}$, we have $\im E\supseteq\im[I_n\ 0_{n,n+m}]^\top$, and then $\Sigma=\SigmaR$ due to Theorem~\ref{thm:sys_set}.

\emph{Remark:}
Theorem~\ref{thm:sys_set} shows that the condition $(A,B)\in\SigmaN$ is ignorable if \eqref{perterbation_condition} holds.
However, if it does not hold, designing a controller without the condition $(A,B)\in\SigmaN$ may lead to conservatism.
Consider the case where $E = e_{1,2n+m}$ and $\hPhi_{22}$ is singular.
Then, noting that $[I\ A\ B]E=e_{1,n}$, it follows from \eqref{prf_SigmaN} that $(A,B)\in\SigmaN$ is equivalent to
\begin{equation}\label{rem_nontrivialexample}
    [0_{n-1,1}\ I_{n-1}](X_+-AX-BU)(I-\hPhi_{22}\hPhi_{22}^\dagger)=0.
\end{equation}
Clearly, $\SigmaN$ does not cover the entire space but imposes linear constraints on the system matrices.
However, these constraints are challenging to characterize using QMIs.

\subsection{LMI Characterization of Data Informativity}
From Definition~\ref{dfn:datainfo_qstab}, data $(X_+,X,U)$ is informative for quadratic stabilization when the Lyapunov inequality~\eqref{qstab} holds for all $(A, B)\in\Sigma$ with a common $P\succ 0$ and $K$.
The solutions to the Lyapunov inequality~\eqref{qstab} can be characterized by a QMI:
\begin{equation}\label{qmi_qstab}
    [I_n\ A\ B]M[I_n\ A\ B]^\top\succ 0,
\end{equation}
where
\begin{equation}\label{dfn_M_qstab}
    M\coloneqq\sqmat{P}{0}{0}{-\matvec{I_n}{K}P\matvec{I_n}{K}^\top}.
\end{equation}
Note that $M\in\mathbf{\Pi}_{n,n+m}$ from $P\succ 0$.
Both the set of systems satisfying the Lyapunov inequality and the set $\SigmaR (\supseteq\Sigma)$ are characterized by the QMIs.
Therefore, 
\begin{equation}\label{inclusion_relation}
    \Zab{n}{n+m}(N)\subseteq\Zab{n}{n+m}^+(M)
\end{equation} 
with some $P\succ 0$ and $K$ is a sufficient condition for data informativity for quadratic stabilization.
Note that this condition is also necessary when $\Sigma=\SigmaR$ holds.
In the existing study that treats exogenous disturbance~\cite{DDCQMI:Waarde2023_siam_qmi}, the matrix S-lemma given by Proposition~\ref{prop:Slem_beta} transforms the inclusion relationship into an LMI, and in the study that treats measurement noise~\cite{DDCQMI:BISOFFI2024_CSL}, Petersen's lemma~\cite{Con:PETERSEN1987,DDCQMI:BISOFFI2022_Petersen} plays a similar role.

Nevertheless, those propositions cannot be applied to our problem in a straightforward manner.
They require the system set $\SigmaR$ to be a matrix ellipsoid, i.e., $N\in\mathbf{\Pi}_{n,n+m}$, which is not generally satisfied under data perturbation.
To illustrate this, consider the energy bounded measurement noise case~\cite{DDCQMI:BISOFFI2024_CSL} without the SNR assumption.
Let $E=I_{2n+m}$ and $\hPhi=\blkdiag(\Theta,-I_T)$ with $\Theta\succeq 0$.
Then, the matrix $N$ defined in~\eqref{dfn_N} reduces to
\begin{equation}
N=\Theta - \matvec{X_+}{-Z}\matvec{X_+}{-Z}^\top,
\end{equation}
where $Z=[X^\top\ U^\top]^\top$.
Consequently, the bottom-right block of $N$ is given by $N_{22}=\Theta_{22}-ZZ^\top$, where $\Theta_{22}\inX{S}{n+m}$ is the corresponding submatrix of $\Theta$.
This implies that if $\Theta_{22}\succ 0$, the condition $N_{22}\preceq 0$ fails to hold for data $Z$ with a sufficiently small norm.
Thus, generally $N\notin\mathbf{\Pi}_{n,n+m}$.
Note that, as implied by Lemma~\ref{lem:NinPi}, the stabilization problem to be addressed later becomes infeasible if this condition is not met.
However, a key advantage of the proposed method is that it does not impose $N \in \mathbf{\Pi}_{n,n+m}$ as a priori assumption, which enables us to assess data informativity precisely.

\emph{Remark:}
We briefly review how the condition $N\in\mathbf{\Pi}_{n,n+m}$ is met in existing literature.
For disturbance case in~\cite{DDCQMI:Waarde2022_TAC_origin,DDCQMI:Waarde2023_siam_qmi}, the condition holds automatically.
Setting $E=[I_n\ 0_{n,n+m}]^\top$ yields $N_{22} = [0_{n+m,n}\ -Z]\hPhi[0_{n+m,n}\ -Z]^\top =Z\hPhi_{22}Z^\top$ from~\eqref{dfn_N}.
Given that $\hPhi_{22}\preceq0$, we obtain $N_{22}\preceq 0$.
The other conditions for $N\in\mathbf{\Pi}_{n,n+m}$ can also be readily verified in~\cite[Sec. IV]{DDCQMI:Waarde2022_TAC_origin} and \cite[Sec. 5.1]{DDCQMI:Waarde2023_siam_qmi}.
In the measurement noise setting~\cite{DDCQMI:BISOFFI2024_CSL}, we recall that $N_{22}=\Theta_{22}-ZZ^\top$.
In this case, the condition $N_{22}\prec 0$ is guaranteed by the SNR assumption~\eqref{eiv:assumption}, which implies $N\in\mathbf{\Pi}_{n,n+m}$.

To resolve this issue, we derive a novel matrix S-lemma that does not require $N\in\mathbf{\Pi}_{n,n+m}$:
\begin{theorem}\label{thm:Slem_new}
    Let $\Omega,\Pi\inX{S}{q+r}$.
    Assume that $\Zqr(\Pi)$ is nonempty, $\Omega_{22}\preceq 0$, and $\ker \Omega_{22}\subseteq\ker \Omega_{12}$.
    Then $\Zqr(\Pi)\subseteq\Zqr^+(\Omega)$ holds if and only if there exist scalars $\alpha\geq 0$ and $\beta>0$ such that
   $\Omega-\alpha \Pi\succeq \blkdiag(\beta I_q,0_r)$.
\end{theorem}

Theorem~\ref{thm:Slem_new} and Proposition~\ref{prop:Slem_beta} derive the same LMI condition, but their preconditions are different.
Unlike Proposition~\ref{prop:Slem_beta}, Theorem~\ref{thm:Slem_new} does not require $\Pi\in\mathbf{\Pi}_{q,r}$, corresponding to $N\in\mathbf{\Pi}_{n,n+m}$ in our problem, but instead additionally requires $\ker \Omega_{22}\subseteq\ker \Omega_{12}$.
Note that nonemptiness of $\Zqr(\Pi)$ is a necessary condition for $\Pi\in\mathbf{\Pi}_{q,r}$~\cite[Theorem 3.2]{DDCQMI:Waarde2023_siam_qmi}.
In this sense, Theorem~\ref{thm:Slem_new} requires a weaker condition for $\Pi$, which corresponds to the set of systems consistent with data, but a stronger condition for $\Omega$, which corresponds to the control objective, than Proposition~\ref{prop:Slem_beta}. 
The following lemma is the key to proving Theorem~\ref{thm:Slem_new}.

\begin{lemma}\label{lem:Ninfo}
   Let $\Omega,\,\Pi\inX{S}{q+r}$.
   Assume that $\Zqr(\Pi)\subseteq\Zqr(\Omega)$, $\Zqr(\Pi)$ is nonempty, $\Omega$ has at least one negative eigenvalue, and $\Omega\in\mathbf{\Pi}_{q,r}$.
   Then $\Pi\in\mathbf{\Pi}_{q,r}$.
\end{lemma}
\begin{proof}
   We define $\mathcal{F}_{\Omega}(Z)$ and $\mathcal{F}_{\Pi}(Z)$ as
   \begin{align*}
      \mathcal{F}_{\Omega}(Z)&\coloneqq [I_q\ Z^{\top}]\Omega[I_q\ Z^{\top}]^{\top}\\
      &=\Omega|\Omega_{22}+(Z+\Omega_{22}^\dagger \Omega_{21})^\top \Omega_{22}(Z+\Omega_{22}^\dagger \Omega_{21}),\\
      \mathcal{F}_{\Pi}(Z)&\coloneqq [I_q\ Z^{\top}]\Pi[I_q\ Z^{\top}]^{\top}\\
      &=\Pi_{11}+\Pi_{12}Z+Z^\top \Pi_{21}+Z^\top \Pi_{22}Z,
   \end{align*}
   respectively, where we have used \eqref{qmi:tenkai}.
   The condition $\Zqr(\Pi)\subseteq\Zqr(\Omega)$ is equivalent to
   \begin{equation}\label{include}
      \mathcal{F}_{\Pi}(Z)\succeq 0\Rightarrow \mathcal{F}_{\Omega}(Z)\succeq 0.
   \end{equation}
    Due to nonemptiness of $\Zqr(\Pi)$, there exists a matrix $Z_0\inX{R}{r\times q}$ such that $\mathcal{F}_{\Pi}(Z_0)\succeq 0$.
   Then, $\mathcal{F}_{\Omega}(Z_0)\succeq 0$ also holds.

   First, we show that
   \begin{equation}\label{MN}
      x^\top \Omega_{22}x<0\Rightarrow x^\top \Pi_{22}x<0.
   \end{equation}
   Suppose, for the sake of contradiction, that there exists $x\inX{R}{r}$ such that $x^\top \Omega_{22}x<0$ and $x^\top \Pi_{22}x\geq 0$.
   Define $Z'\coloneqq Z_0+xy^\top$ with an arbitrary $y\inX{R}{q}$.
   Then, we obtain
   \begin{align*}
      \mathcal{F}_{\Omega}(Z')&= \mathcal{F}_{\Omega}(Z_0)+(x^\top \Omega_{22}x)yy^\top+M_1 xy^\top+yx^\top M_1^\top,\\
      \mathcal{F}_{\Pi}(Z')&= \mathcal{F}_{\Pi}(Z_0)+(x^\top \Pi_{22}x)yy^\top+M_2 xy^\top+yx^\top M_2^\top,
   \end{align*}
   where $M_1\coloneqq(Z_0+\Omega_{22}^\dagger \Omega_{21})^\top \Omega_{22}$ and $M_2\coloneqq \Pi_{12}+Z_0^\top \Pi_{22}$.
   When $M_2 x=0$, we obtain $\mathcal{F}_{\Omega}(Z')\succeq 0$ for all $y\inX{R}{q}$ because \eqref{include} holds and $\mathcal{F}_{\Pi}(Z')\succeq \mathcal{F}_{\Pi}(Z_0)\succeq 0$ for all $y\inX{R}{q}$.
   However, since $x^\top \Omega_{22}x<0$, the range of $y$ satisfying $\mathcal{F}_{\Omega}(Z')\succeq 0$ is bounded, which contradicts the arbitrariness of $y$.
   When $M_2 x\neq 0$, we set $y$ to $y=aM_2 x$ with an arbitrary $a\geq 0$.
   Then, we obtain $\mathcal{F}_{\Omega}(Z')\succeq 0$ for all $a\geq 0$ because \eqref{include} holds and $\mathcal{F}_{\Pi}(Z')\succeq \mathcal{F}_{\Pi}(Z_0)\succeq 0$ for all $a\geq 0$.
   However, the range of $a$ satisfying $\mathcal{F}_{\Omega}(Z')\succeq 0$ is also bounded, which contradicts the arbitrariness of $a$.
   Thus, we obtain \eqref{MN}.

   Next, we prove that $\Pi_{22}\preceq 0$.
   It suffices to show that 
   \begin{equation}\label{MN0}
      x^\top \Omega_{22}x=0\Rightarrow x^\top \Pi_{22}x\leq0.
   \end{equation}
   Indeed, since $\Omega_{22}\preceq 0$, the conditions \eqref{MN} and \eqref{MN0} together imply $\Pi_{22}\preceq 0$.
   Define $f(x)\coloneqq x^\top \Pi_{22}x$.
   Let $x_1\inX{R}{r}$ be a vector satisfying $x_1^\top\Omega_{22}x_1=0$.
   Note that $\Omega_{22}x_1=0$ also holds because $\Omega_{22}\preceq 0$.
   Since $\Omega_{22}$ has a negative eigenvalue, there exists $x_2$ such that $x_2^\top \Omega_{22}x_2<0$.
   Now, consider the vector $x_\epsilon\coloneqq x_1+\epsilon x_2$ for any $\epsilon\neq 0$.
   A direct calculation yields $x_\epsilon^\top \Omega_{22}x_\epsilon=\epsilon^2 x_2^\top\Omega_{22}x_2<0$.
   Thus, from \eqref{MN}, we have $f(x_\epsilon)<0$.
   By the continuity of $f$, taking $\epsilon\to0$ yields $f(x_1) \leq 0$.
   Thus, \eqref{MN0} is proved, which implies that $\Pi_{22}\preceq 0$.

   Next, we show that $\ker \Pi_{22}\subseteq \ker\Pi_{12}$.
   Let $x_1\inX{R}{r}$ be a vector satisfying $\Pi_{22}x_1=0$.
   We will prove that $\Pi_{12}x_1=0$.
   Note that $\Omega_{22}x_1=0$ also holds due to $\Omega_{22}\preceq 0$ and the contrapositive of \eqref{MN}.
   Since $\Omega_{22}$ has a negative eigenvalue, there exists $x_2$ such that $x_2^\top \Omega_{22}x_2 <0$.
   Define $Z'\coloneqq Z_0+x_1y_1^\top+x_2y_2^\top$ with arbitrary $y_1,y_2\inX{R}{q}$.
   Then, we obtain
   \begin{align*}
      \mathcal{F}_{\Omega}(Z')\!=& \mathcal{F}_{\Omega}(Z_0)+(x_1y_1^\top+x_2y_2^\top)^\top \Omega_{22}(x_1y_1^\top+x_2y_2^\top)\\
      &+M_1(x_1y_1^\top+x_2y_2^\top)+(x_1y_1^\top+x_2y_2^\top)^\top M_1^\top\\
      =&\mathcal{F}_{\Omega}(Z_0)+\!(x_2^\top \Omega_{22}x_2)y_2y_2^\top\!+\!M_1 x_2y_2^\top\!+\!y_2x_2^\top M_1^\top\!,
    \end{align*}
    \begin{align*}
      \mathcal{F}_{\Pi}(Z')=&\mathcal{F}_{\Pi}(Z_0)+(y_1x_1^\top+y_2x_2^\top)\Pi_{22}(x_1y_1^\top+x_2y_2^\top)\\
      &+M_2(x_1y_1^\top+x_2y_2^\top)+(x_1y_1^\top+x_2y_2^\top)^\top M_2^\top\\
      =&\mathcal{F}_{\Pi}(Z_0)+py_2y_2^\top\\
      &+\Pi_{12}x_1y_1^\top\!+y_1x_1^\top \Pi_{21}+M_2 x_2y_2^\top\!+y_2x_2^\top M_2^\top\!,
   \end{align*}
   where $p=x_2^\top \Pi_{22}x_2$. 
   Note that $p<0$ due to \eqref{MN} and $x_2^\top \Omega_{22}x_2<0$.
   We set $y_1$ to $y_1=ay$ with $y=\Pi_{12}x_1$ and an arbitrary $a\geq 0$, and $y_2$ to $y_2=-\frac{1}{p}M_2 x_2+b y$ with an arbitrary $b\in\mathbb{R}$.
   Then, we obtain
   \begin{equation*}
      \mathcal{F}_{\Pi}(Z')=\mathcal{F}_{\Pi}(Z_0)+(2a+pb^2)yy^\top -\frac{1}{p}(M_2 x_2)(M_2 x_2)^\top.
   \end{equation*}
   Owing to the arbitrariness of $a\geq 0$, for all $b$, there exists $a\geq 0$ such that $(2a+pb^2)\geq 0$,
   and then $\mathcal{F}_{\Pi}(Z')\succeq \mathcal{F}_{\Pi}(Z_0)\succeq 0$ holds.
   Due to \eqref{include}, for all $b$, there exists $a\geq 0$ such that $\mathcal{F}_{\Omega}(Z')\succeq 0$.
   If we assume that $y\neq 0$, the range of $b$ satisfying $\mathcal{F}_{\Omega}(Z')\succeq 0$ is bounded because $x_2^\top \Omega_{22}x_2<0$, which contradicts the arbitrariness of $b$.
   Therefore, $y = \Pi_{12}x_1 =0$ holds, which implies that $\ker\Pi_{22}\subseteq \ker\Pi_{12}$.

   Finally, we show $\Pi|\Pi_{22}\succeq 0$.
   Since $\ker \Pi_{22}\subseteq \ker \Pi_{12}$, we have $\mathcal{F}_{\Pi}(Z)=\Pi|\Pi_{22}+(Z+\Pi_{22}^\dagger \Pi_{21})^\top \Pi_{22}(Z+\Pi_{22}^\dagger \Pi_{21})$
   by utilizing \eqref{qmi:tenkai}.
   Since $\Pi_{22}\preceq 0$ and $\mathcal{F}_{\Pi}(Z_0)\succeq 0$, we obtain $\Pi|\Pi_{22}\succeq 0$.
\end{proof}

Lemma~\ref{lem:Ninfo} claims that if $\Zqr(\Omega)$ is a matrix ellipsoid and contains $\Zqr(\Pi)$, $\Zqr(\Pi)$ is also a matrix ellipsoid.
Next, we prove Theorem~\ref{thm:Slem_new} by leveraging Lemma~\ref{lem:Ninfo}.

\begin{proof}
   (Theorem~\ref{thm:Slem_new})
   We first prove the ``if'' statement.
   Assume that there exist scalars $\alpha\geq 0$ and $\beta>0$ such that $\Omega-\alpha \Pi\succeq \blkdiag(\beta I_q, 0)$.
   Then, when $Z\in\Zqr(\Pi)$, we obtain
   \begin{equation*}
      \IqZ^\top \Omega\IqZ\succeq \alpha\IqZ^\top \Pi\IqZ +\beta I_q\succ 0,
   \end{equation*}
   which implies $Z\in\Zqr^+(\Omega)$.
   Thus, $\Zqr(\Pi)\subseteq\Zqr^+(\Omega)$ holds.

   Second, we prove the ``only if'' statement.
   Assume that $\Zqr(\Pi)\subseteq\Zqr^+(\Omega)$.
   Since $\Zqr(\Pi)$ is nonempty, there exists $Z_0\in\Zqr(\Pi)$, which is also contained in $\Zqr^+(\Omega)$.
   By utilizing \eqref{N_bunkai}, we have that
   \begin{equation}\label{prf_Msucc0}
      \Omega|\Omega_{22}\succeq \matvec{I_q}{Z_0}^\top \Omega\matvec{I_q}{Z_0}\succ 0.
   \end{equation}
   When $\Omega_{22}=0$, $\Omega_{12}=0$ is obtained by $\ker \Omega_{22}\subseteq \ker \Omega_{12}$, and then $\Omega_{11}=\Omega|\Omega_{22}\succ 0$ also holds from \eqref{prf_Msucc0}.
   Thus, $\alpha=0$ and $0<\beta<\lambda_{\mathrm{min}}$ satisfy $\Omega-\alpha \Pi\succeq \blkdiag(\beta I_q, 0)$ where $\lambda_{\mathrm{min}}$ is the minimum eigenvalue of $\Omega_{11}$.
   When $\Omega_{22}\neq 0$, $\Omega_{22}$ has at least one negative eigenvalue.
   Since the preconditions and \eqref{prf_Msucc0} lead to $\Omega\in\mathbf{\Pi}_{q,r}$, we can apply Lemma~\ref{lem:Ninfo}, and then we have $\Pi\in\mathbf{\Pi}_{q,r}$.
   Thus, we obtain $\Omega-\alpha \Pi\succeq \blkdiag(\beta I_q, 0)$ by applying Proposition~\ref{prop:Slem_beta}.
\end{proof}

We obtain an LMI condition equivalent to the data informativity by applying Theorem~\ref{thm:Slem_new} for $M$ and $N$ as defined in \eqref{dfn_M_qstab} and \eqref{dfn_N} respectively.
\begin{theorem}\label{thm:qstab}
    Suppose that the data $(X_+,X,U)$ is generated by \eqref{eq:cleandata} and \eqref{eq:measured_data} under Assumption~\ref{ass:noiseset}.
    Then, the data $(X_+,X,U)$ is informative for quadratic stabilization 
    if there exist a positive definite $P\inX{S}{n}$, a matrix $L\inX{R}{m\times n}$, and scalars $\alpha\geq 0$ and $\beta>0$ such that
    \begin{equation}\label{LMI_result}
    \mathcal{M}(P,L,\beta)-\alpha\sqmat{N}{0}{0}{0}\succeq 0,
    \end{equation}
    where 
    \begin{equation}\label{dfn_MPLbeta}
    \mathcal{M}(P,L,\beta) = 
        \begin{bmatrix}
        P-\beta I_n&0 &0      &0\\
        0          &-P&-L^\top&0\\
        0          &-L&0_m    &L\\
        0          &0 &L^\top &P
        \end{bmatrix},
    \end{equation}
    and $N$ is defined as \eqref{dfn_N}.
    In particular, when the perturbation condition \eqref{perterbation_condition} holds, \eqref{LMI_result} is a necessary and sufficient condition.
    Moreover, if \eqref{LMI_result} is feasible, the controller $K=LP^{-1}\inX{R}{m\times n}$ stabilizes any system in $\Sigma$.
\end{theorem}
\begin{proof}
    First, we confirm that Theorem~\ref{thm:Slem_new} can be applied to $M$ and $N$, where $M$ is defined as \eqref{dfn_M_qstab}.
    The system set $\Sigma$ is nonempty because it contains the true system $(A^*,B^*)$.
    Thus, $\SigmaR$ corresponding to $\Zab{n}{n+m}(N)$ is also nonempty due to $\Sigma\subseteq\SigmaR$.
    On the other hand, the matrix $M$ satisfies $M_{22}\preceq 0$ and $\ker M_{22}\subseteq\ker M_{12}$ when $P\succ 0$.

    Recall that the data is informative for quadratic stabilization if \eqref{inclusion_relation} holds with $P\succ 0$ and $K$, and the converse also holds when $\Sigma=\SigmaR$ holds.
    From Theorem~\ref{thm:Slem_new}, it follows that \eqref{inclusion_relation} holds if and only if there exist scalars $\alpha\geq 0$ and $\beta>0$ such that
    $M-\alpha N\succeq \blkdiag(\beta I_n,0_{n+m})$.
    By setting $L=KP$ and applying Schur complement to its bottom-right block, this inequality is equivalently converted into \eqref{LMI_result}.
    Therefore, the data is informative if there exist $P\succ 0$, $L$, $\alpha\geq 0$ and $\beta>0$ such that \eqref{LMI_result} holds.
    Moreover, when $\im E\supseteq\im[I_n\ 0_{n,n+m}]^\top$ or $\hPhi_{22}\prec 0$, from Theorem~\ref{thm:sys_set}, it follows that $\Sigma=\SigmaR$.
    Thus, in this case, the condition is also necessary.
    Note that controller is obtained as $K=LP^{-1}$ since we set $L=KP$.
\end{proof}

Theorem~\ref{thm:qstab} provides a necessary and sufficient condition for data informativity under the perturbation condition \eqref{perterbation_condition}. 
In particular, our analysis does not impose any prior assumptions on the data, demonstrating that the LMI condition \eqref{LMI_result} is the sole requirement imposed on the data under \eqref{perterbation_condition}.
Note that Theorem~\ref{thm:qstab} recovers the result under exogenous disturbances~\cite[Theorem 5.1(a)]{DDCQMI:Waarde2023_siam_qmi} and generalizes the result under measurement noise~\cite[Theorem 1]{DDCQMI:BISOFFI2024_CSL}, since one of the perturbation conditions $\im E\supseteq\im[I_n\ 0_{n,n+m}]^\top$ holds in both cases.

\subsection{Informativity of Rank-deficient Data}\label{subsec:rank-deficient}

The above analysis relaxes the large SNR assumption \eqref{eiv:assumption} in the measurement noise case.
While this relaxation allows us to determine the informativity of the low-SNR data, it remains unclear whether such data can actually be informative.
The following lemma shows limitations of the informative data.

\begin{lemma}\label{lem:NinPi}
    Suppose that
    there exist $P\succ 0,L,\alpha\geq 0$, and $\beta>0$ such that \eqref{LMI_result} holds.
    Then, $N\in\mathbf{\Pi}_{n,n+m}$ holds.
\end{lemma}
\begin{proof}
    When there exist $P\succ 0,L,\alpha\geq 0$, and $\beta>0$ such that \eqref{LMI_result} holds, $\Zab{n}{n+m}(N)\subseteq \Zab{n}{n+m}^+(M)$ where $K\coloneqq LP^{-1}$ and $M$ is defined as \eqref{dfn_M_qstab}.
    Since $\Zab{n}{n+m}^+(M)\subseteq \Zab{n}{n+m}(M)$, $\Zab{n}{n+m}(N)\subseteq \Zab{n}{n+m}(M)$ holds.
    Moreover, $\Zab{n}{n+m}(N)$ is nonempty since the true system belongs to $\SigmaR$.
    Then, by applying Lemma~\ref{lem:Ninfo} we obtain $N\in\mathbf{\Pi}_{n,n+m}$.
\end{proof}

Lemma~\ref{lem:NinPi} states that data informativity for quadratic stabilization requires $N\in\mathbf{\Pi}_{n,n+m}$. 
Recall that the SNR condition \eqref{eiv:assumption} implies $N_{22}\prec 0$. 
Since the condition $N\in\mathbf{\Pi}_{n,n+m}$ includes $N_{22}\preceq 0$, it can be interpreted as a relaxed SNR condition.
By relaxing the strict inequality, we can handle, for example, rank-deficient data matrices.
For demonstration, we take $\hPhi\coloneqq \blkdiag(\hat{\Theta},-I_T)$ with $\hat{\Theta}\succ 0$ for the noise set $\mathcal{D}$.
Then, the
system set $\Sigma$ corresponds to $\Zab{n}{n+m}(N)$ by Theorem~\ref{thm:sys_set}.
By the definition of $N$, we obtain
$N_{22}=E_-\hat{\Theta}E_-^\top-ZZ^\top$,
where $E_-\inX{R}{(n+m)\times \hat{n}}$ is the lower submatrix of $E=[E_+^\top\ E_-^\top]^\top$ and $Z=[X^\top U^\top]^\top$.
When $E_-$ does not have full row rank, the condition $N_{22}\preceq 0$ does not require the data matrix $Z$ to have full row rank, whereas $N_{22}\prec 0$ does.
In Sec.~\ref{subsec:muri}, we provide a numerical example where $Z=[X^\top\ U^\top]^\top$ does not have full row rank but~\eqref{LMI_result} is feasible.

Data can be informative even when $N_{22}\not\prec 0$.
However, when $N_{22}\not\prec 0$ and \eqref{LMI_result} is feasible, the left-hand side of \eqref{LMI_result} has eigenvalues equal to zero, because its bottom-right $(2n+m) \times (2n+m)$ submatrix must become singular to satisfy the positive semidefinite constraint.
The inevitable presence of such zero eigenvalues implies that Slater's condition fails, leading to potential numerical instability~\cite{Math:Drusvyatskiy_FacialReduction}.
To avoid this issue,
we provide a numerically stable alternative formulation of \eqref{LMI_result}.
\begin{corollary}\label{cor:interior}
    For $N$ in~\eqref{dfn_N},
    let $V \in \mathbb{R}^{(n+m) \times \mathrm{rank}\,N_{22}}$ be a matrix whose columns form an orthonormal basis of $\mathrm{im}\,N_{22}$.
    Let $V=[V_+^\top\ V_-^\top]^\top$, where $V_+\inX{R}{n\times\mathrm{rank}\,N_{22}},V_-\inX{R}{m\times\mathrm{rank}\,N_{22}}$.
    Define
    $\bar{N}\coloneqq
    \smallsqmat{I}{0}{0}{V}^\top N\smallsqmat{I}{0}{0}{V}$.
    If $N\in\mathbf{\Pi}_{n,n+m}$ and there exist a positive definite $\bar{P}\inX{S}{n}$, a matrix $\bar{Y}\inX{R}{\mathrm{rank}\,N_{22}\times n}$, and a scalar $\bar{\alpha}> 0$ such that
    \begin{subequations}\label{LMI_result_cor}
    \begin{align}
    &\begin{bmatrix}
        \bar{P} &0                        &0      \\
        0       &0_{\mathrm{rank}\,N_{22}}&\bar{Y}\\
        0       &\bar{Y}^\top             &\bar{P}
      \end{bmatrix}
      -\bar{\alpha}\begin{bmatrix}
        \bar{N}&{0}\\{0}&{0_n}
      \end{bmatrix}\succ 0, \label{LMI_result_interior}\\
      &\bar{P} = V_+\bar{Y}  \label{result_eq},
    \end{align}
    \end{subequations}
    there exists $\beta>0$ such that \eqref{LMI_result} holds with $P= \bar{P}, L= V_-\bar{Y}$, and $\alpha=\bar{\alpha}$.
    Conversely, if there exist $P\succ 0,L,\alpha\geq 0$, and $\beta>0$ such that \eqref{LMI_result} holds, $N\in\mathbf{\Pi}_{n,n+m}$ and there exists $\bar{\alpha}\geq0$ such that \eqref{LMI_result_cor} holds with $\bar{P}= P$ and $Y=V^\top [{P}\ {L}^\top]^\top$.
\end{corollary}
\begin{proof}
    We prove the first statement.
    Assume that $N\in\mathbf{\Pi}_{n,n+m}$ and there exist $\bar{P}\succ 0, \bar{Y}$, and $\bar{\alpha}> 0$ such that \eqref{LMI_result_interior}, \eqref{result_eq} hold.
    By applying Schur complement, we rewrite \eqref{LMI_result_interior} as
    $\blkdiag(\bar{P},-\bar{Y}\bar{P}^{-1}\bar{Y}^\top)-\bar{\alpha}\bar{N}\succ0$.
    Then, there exists $\beta>0$ such that
    $\blkdiag(\bar{P}-\beta I,-\bar{Y}\bar{P}^{-1}\bar{Y}^\top)-\bar{\alpha}\bar{N}\succeq0$.
    Applying a congruent transformation with $\blkdiag({I_n},{V^\top})$, we obtain
    $\blkdiag(\bar{P}-\beta I,-V\bar{Y}\bar{P}^{-1}\bar{Y}^\top V^\top)-\bar{\alpha}{N}\succeq0$,
    using
    $N=\smallsqmat{I}{0}{0}{V}\bar{N}\smallsqmat{I}{0}{0}{V}^\top$
    because $\ker N_{22}\subseteq\ker N_{12}$ from $N\in\mathbf{\Pi}$.
    By setting $P\coloneqq \bar{P},L\coloneqq V_-\bar{Y}$, and $\alpha\coloneqq\bar{\alpha}$ and using \eqref{result_eq}, we obtain
    \begin{equation}\label{prf_schur}
        \blkdiag\left(P-\beta I,-\matvec{P}{L}{P}^{-1}\matvec{P}{L}^\top\right)-{\alpha}N\succeq 0.
    \end{equation}
    Then, applying Schur complement, we obtain \eqref{LMI_result}.

    Next, we prove the second statement.
    Assume that there exist $P\succ 0,L,\alpha\geq 0$, and $\beta>0$ such that \eqref{LMI_result} holds.
    Then, from Lemma~\ref{lem:NinPi}, $N\in\mathbf{\Pi}_{n,n+m}$.
    Next, we derive \eqref{LMI_result_interior}.
    Applying Schur complement, we rewrite \eqref{LMI_result} as \eqref{prf_schur}.
    By a congruent transformation with {$\blkdiag({I_n},{V})$}, we obtain
    \begin{equation}\label{prf_schur_y}
      \blkdiag(\bar{P}-\beta I, -\bar{Y}\bar{P}^{-1}\bar{Y}^\top)-\alpha\bar{N}\succeq 0,
    \end{equation}
    where we define $\bar{Y}\coloneqq V^\top[P\ L^\top]^\top$ and $\bar{P}\coloneqq P$.
    Decomposing
    $\bar{N}=\left[\begin{smallmatrix}
        \bar{N}_{11} & \bar{N}_{12}\\
        \bar{N}_{21} & \bar{N}_{22}
    \end{smallmatrix}\right],$
    with $\bar{N}_{11}\inX{S}{n},$ and $\bar{N}_{22}\inX{S}{\mathrm{rank}\,N_{22}},$
    we note that $\bar{N}_{22}$ is negative definite from the definition of $V$.
    Thus, there exists $\delta_\alpha>0$ such that $\beta I-\delta_\alpha\bar{N}|\bar{N}_{22}\succ 0$ and $-\delta_\alpha \bar{N}_{22}\succ 0$.
    By Schur complement, we obtain
    \begin{equation}\label{prf_deltaalpha}
        \blkdiag(\beta I_n,0_{\rm rank\,N_{22}})-\delta_\alpha \bar{N}\succ 0.
    \end{equation}
    From \eqref{prf_schur_y} and \eqref{prf_deltaalpha}, it follows that
    $\blkdiag(\bar{P},-\bar{Y}\bar{P}^{-1}\bar{Y}^\top)-(\alpha+\delta_\alpha)\bar{N}\succ 0$.
    Setting $\bar{\alpha}\coloneqq\alpha+\delta_\alpha$ and applying Schur complement, we obtain \eqref{LMI_result_interior}.
    Finally, we derive \eqref{result_eq}.
    From the bottom-right block of \eqref{prf_schur}, we obtain
    $\alpha(-N_{22})-\left[
    \begin{smallmatrix}
        P\\ L
    \end{smallmatrix}\right]P^{-1}\left[
    \begin{smallmatrix}
        P\\ L
    \end{smallmatrix}\right]^\top\succeq 0.$
    Then, from~\cite[Lemma A.1]{DDCQMI:Waarde2023_siam_qmi}, there exists $M\inX{R}{(n+m)\times n}$ such that {$[P\ L^\top]^\top P^{-1/2}=\alpha^{1/2}(-N_{22})^{1/2}M$}.
    This equation implies {$\im[P\ L^\top]^\top\subseteq\im V$} because $P$ is nonsingular and $\im N_{22}=\im V$.
    Then, we obtain {$V\bar{Y}=VV^\top[P\ L^\top]^\top=[P\ L^\top]^\top$}, which includes \eqref{result_eq}.
\end{proof}

Note that $N\in\mathbf{\Pi}_{n,n+m}$ can be easily verified since $N$ is defined by the known matrices $E$, $\hPhi$, and the data $(X_+,X,U)$.
Corollary~\ref{cor:interior} provides a numerically stable alternative to Theorem~\ref{thm:qstab} because \eqref{LMI_result_cor} consists of linear equations and strict inequalities, avoiding the zero eigenvalues.
Moreover, since a solution of \eqref{LMI_result} can be obtained from the solution of \eqref{LMI_result_cor}, Corollary~\ref{cor:interior} can be used for controller synthesis.

\section{Data Informativity for Optimal Control under Data Perturbation}\label{sec:optimal}
In this section, we address the data informativity problem for guaranteeing the $\Hc{2}$ and $\Hc{\infty}$ performance by leveraging the findings in the previous section.

\subsection{$\Hc{2}$ control}
We consider a linear system
$x_+=A^*x+B^*u+w,z=Cx+Du$
with the exogenous input $w$ and the performance output $z$.
We assume that $(A^*,B^*)$ are unknown but $(C,D)$ are known as in~\cite{DDCQMI:Waarde2022_TAC_origin,DDCQMI:Steentjes2022_Cont_Sys_Let_Covariance,DDCQMI:Hu2022_CDC}.
As in the previous section, we consider a setting where the $T$-long offline batch data $(X_+,X,U)$, subject to data perturbation satisfying Assumption~\ref{ass:noiseset}, is available.
Let $G(z)\coloneqq (C+DK)(zI-(A+BK))^{-1}$ be the transfer function of the closed-loop system for $(A,B)$ from $w$ to $z$ with the state-feedback control $u=Kx$.
Our objective is to render the $\Hc{2}$ and $\Hc{\infty}$ norms of $G(z)$ less than a prescribed constant $\gamma>0$ for all $(A,B)$ in the set of systems consistent with the data, i.e., $\Sigma$.
The $\Hc{2}$ norm of $G(z)$ is less than $\gamma$ if and only if there exists a matrix $P\succ 0$ satisfying the following conditions~\cite{DDCQMI:Waarde2022_TAC_origin}:
\begin{subequations}
  \begin{align}
   &   P\succ A_K^\top P A_K+C_K^\top C_K\label{opt:h2_1},\\
   &  \tr P<\gamma^2\label{opt:h2_2},
  \end{align}
\end{subequations}
where $A_K\coloneqq A+BK$ and $C_K\coloneqq C+DK$.
Based on this, we define data informativity for $\mathcal{H}_{2}$ control.
\begin{definition}\label{dfn:datainfo_h2}
  The data $(X_+,X,U)$ is informative for $\mathcal{H}_{2}$ control with performance $\gamma>0$ under data perturbation if there exist a positive definite $P\inX{S}{n}$ and a controller $K\inX{R}{m\times n}$ such that \eqref{opt:h2_1}, \eqref{opt:h2_2} hold for any $(A,B)\in\Sigma$.
\end{definition}

From~\cite{DDCQMI:Waarde2022_TAC_origin}, matrices $P\succ 0$ and $K$ satisfy \eqref{opt:h2_1} and \eqref{opt:h2_2} if and only if $Y\coloneqq P^{-1}$ and $L\coloneqq KP^{-1}$ satisfy
\begin{subequations}\label{opt:h2}
\mathtoolsset{showonlyrefs=false}
  \begin{align} 
    &[I_n\ A\ B]M_\Hc{2}[I_n\ A\ B]^\top\succ 0, \label{opt:h2_qmi}\\
    &Y-{C_{Y,L}}^\top C_{Y,L}\succ 0,   \label{opt:h2_3}\\
    &\tr Y^{-1}<\gamma^2,               \label{opt:h2_4}
  \end{align}
\end{subequations}
where $\Cyl = CY+DL$ and 
\begin{equation}\label{opt:dfn_M_h2}
    M_\Hc{2}\coloneqq\blkdiag\left(Y,-\matvec{Y}{L}(Y-\Cyl^\top\Cyl)^{-1}\matvec{Y}{L}^\top\right).
\end{equation}
Note that $M_\Hc{2}\in\mathbf{\Pi}_{n,n+m}$ when \eqref{opt:h2_3} holds.
From Lemma~\ref{lem:conservative_characterization}, $\Sigma$ is characterized by the QMI with $N$ defined in \eqref{dfn_N}.
Therefore, the data $(X_+,X,U)$ is informative if there exist $Y\succ 0$ and $L$ such that \eqref{opt:h2_3}, \eqref{opt:h2_4} hold and that
\begin{equation}\label{opt:inclusion_relation_h2}
    \Zab{n}{n+m}(N)\subseteq\Zab{n}{n+m}^+(M_{\Hc{2}}).
\end{equation}
Note that this condition is also necessary when $\Sigma=\SigmaR$.
Applying Theorem~\ref{thm:Slem_new} to \eqref{opt:inclusion_relation_h2}, we derive LMI conditions equivalent to data informativity for $\mathcal{H}_2$ control.
\begin{theorem}\label{thm:h2}
    Suppose that
    the data $(X_+,X,U)$ generated by \eqref{eq:cleandata} and \eqref{eq:measured_data} under Assumption~\ref{ass:noiseset}.
    Then, the data $(X_+,X,U)$ is informative for $\mathcal{H}_{2}$ control with performance $\gamma>0$ if there exist
    positive definite matrices $Y,Z\inX{S}{n}$,
    a matrix $L\inX{R}{m\times n}$, and scalars $\alpha\geq 0$ and $\beta>0$ such that
    \begin{subequations}\label{LMI_H2}
    \begin{align}
      &\begin{bmatrix}
        Y-\beta I_n&0&0     &0   &0        \\
        0          &0&0     &Y   &0        \\
        0          &0&0     &L   &0        \\
        0          &Y&L^\top&Y   &\Cyl^\top\\
        0          &0&0     &\Cyl&I_p
      \end{bmatrix}
      -\alpha
      \begin{bmatrix}
        N&{0}\\{0}&{0_{n+p}}
      \end{bmatrix}\succeq 0,\label{h2:lmi1}\\
      &\sqmat{Y}{\Cyl^\top}{\Cyl}{I_p}\succ 0,\label{h2:lmi2}\\
      &\sqmat{Z}{I_n}{I_n}{Y}\succeq 0,\tr Z<\gamma^2,\label{h2:lmi3}
    \end{align}
    \end{subequations}
    where $N$ is defined as \eqref{dfn_N}.
    In particular, when the perturbation condition \eqref{perterbation_condition} holds,
    \eqref{LMI_H2} is a necessary and sufficient condition.
    Moreover, if \eqref{LMI_H2} is feasible, the controller $K=LY^{-1}\inX{R}{m\times n}$ ensures that $\Hc{2}$ norm of $G(z)$ is less than $\gamma$ for any system in $\Sigma$.
\end{theorem}
\begin{proof}
    First, \eqref{opt:h2_3} is equivalently converted into \eqref{h2:lmi2} by applying Schur complement.
    Second, we show that \eqref{opt:h2_4} and \eqref{h2:lmi3} are equivalent.
    If $Y\succ 0$ and $L$ satisfy \eqref{opt:h2_4}, there exists $Z\succ 0$ such that $Z\succeq Y^{-1}$ and $\tr Z<\gamma^2$.
    Applying Schur complement, we obtain \eqref{h2:lmi3}.
    Conversely, if $Y\succ 0$, $Z\succ 0$ and $L$ satisfy \eqref{h2:lmi3}, it follows that $\tr Y\leq\tr Z<\gamma^2$.
    Third, when \eqref{opt:h2_3} holds, we have $M_\Hc{2}\in\mathbf{\Pi}_{n,n+m}$, which allows us to apply Theorem~\ref{thm:Slem_new} to \eqref{opt:inclusion_relation_h2}.
    As a result, $Y\succ0$ and $L$ satisfy \eqref{opt:inclusion_relation_h2} if and only if there exists $\alpha\geq 0$ and $\beta>0$ satisfying $M_\Hc{2}-\alpha N\succeq\blkdiag(\beta I_n, 0_{n+m})$.
    This inequality can be converted into \eqref{h2:lmi1} by applying Schur complement twice.
    Therefore, there exist $Y\succ 0$, $Z\succ 0$, $L$, $\alpha\geq 0$ and $\beta >0$ such that \eqref{LMI_H2} holds if and only if $Y$ and $L$ satisfy \eqref{opt:h2_3}, \eqref{opt:h2_4} and \eqref{opt:inclusion_relation_h2}, which is sufficient for data informativity and necessary if $\Sigma = \SigmaR$.
    Note that $\Sigma = \SigmaR$ follows from Theorem~\ref{thm:sys_set} when $\im E\supseteq\im[I_n\ 0_{n,n+m}]^\top$ or $\hPhi_{22}\prec 0$.
\end{proof}

Using Theorem~\ref{thm:h2}, we can formulate the $\mathcal{H}_2$ optimal control problem with the objective function $J \coloneqq \gamma^2$:
\begin{equation}\label{opt:opt_h2}
  \begin{split}
    \underset{J,Y,Z,L,\alpha,\beta}{\mathrm{minimize}}\ J\ \st
    &Y\succ 0,Z\succ 0,\alpha\geq 0,\beta>0,\\
    &\eqref{h2:lmi1},\eqref{h2:lmi2},\tr Z<J.
  \end{split}
\end{equation}
Although the minimum of the problem \eqref{opt:opt_h2} generally does not exist due to strict inequalities, a suboptimal solution can be found by relaxing them to non-strict inequalities with a small positive margin.

\subsection{$\Hc{\infty}$ control}
The $\Hc{\infty}$ norm of the transfer function $G(z)$ is less than $\gamma>0$ if and only if there exists a positive definite $P\inX{S}{n}$ satisfying the following condition~\cite[Sec. 4.6]{Cntr:scherer2000linear}:
\begin{equation}\label{opt:hinf}
  \begin{bmatrix}
    P   &0         &A_K^\top P&C_K^\top\\
    0   &\gamma I_n&P         &0\\
    PA_K&P         &P         &0\\
    C_K &0         &0         &\gamma I_p
  \end{bmatrix}\succ 0.
\end{equation}
Similar to the $\Hc{2}$ control case, we define data informativity for $\Hc{\infty}$ control as follows:
\begin{definition}\label{dfn:datainfo_hinf}
  The data $(X_+,X,U)$ is informative for $\mathcal{H}_{\infty}$ control with performance $\gamma>0$ under data perturbation if there exist a positive definite $P\inX{S}{n}$ and a controller $K\inX{R}{m\times n}$ such that \eqref{opt:hinf} holds for any $(A,B)\in\Sigma$.
\end{definition}

We characterize the condition \eqref{opt:hinf} by a QMI involving system matrices and an LMI.
We remark that similar characterization is conducted in~\cite{DDCQMI:Steentjes2022_Cont_Sys_Let_Covariance}.
\begin{lemma}
    Matrices $P\succ 0$ and $K$ satisfy \eqref{opt:hinf} if and only if $Y\coloneqq \gamma P^{-1}$ and $L\coloneqq KY$ satisfy
    \begin{subequations}\label{opt:hinflem}
    \begin{align}
      &[I_n\ A\ B]M_\Hc{\infty}[I_n\ A\ B]^\top\succ 0, \label{opt:hinf_qmi}\\
      &\sqmat{Y}{\Cyl^\top}{\Cyl}{\gamma^2 I_p}\succ 0, \label{opt:hinf_yl}
    \end{align}
    \end{subequations}
    where $\Cyl\coloneqq CY+DL$
    and
    \small\begin{equation}\label{opt:dfn_M_hinf}
        M_\Hc{\infty}\coloneqq\blkdiag\left(Y-I_n,-\matvec{Y}{L}(Y-\frac{1}{\gamma^2}\Cyl^\top\Cyl)^{-1}\matvec{Y}{L}^\top\right).
    \end{equation}
\end{lemma}
\begin{proof}
    Assume that there exist $P$ and $K$ such that \eqref{opt:hinf} holds.
    Applying a congruent transformation with $\blkdiag(Y,I_n,Y,I_p)$ and dividing the entire matrix by $\gamma$, we obtain
    \begin{equation}\label{opt:hinf1}
    \begin{bmatrix}
      Y                   &0  &\Ayl^\top&\frac{1}{\gamma}\Cyl^\top\\
      0                   &I_n&I_n      &0\\
      \Ayl                &I_n&Y        &0\\
      \frac{1}{\gamma}\Cyl&0  &0        &I_p
    \end{bmatrix}\succ 0
    \end{equation}
    where $\Ayl\coloneqq AY+BL$.
    By applying Schur complement three times, \eqref{opt:hinf1} is equivalent to $Y-\Cyl^\top\Cyl/\gamma^2\succ 0$ and 
      $Y-\Ayl(Y-\Cyl^\top\Cyl/\gamma^2)^{-1}\Ayl^\top-I_n\succ 0$.
    The former inequality can be transformed into \eqref{opt:hinf_yl} by Schur complement and the latter is equivalent to \eqref{opt:hinf_qmi}. 
\end{proof}

Note that $M_\Hc{\infty}\in\mathbf{\Pi}_{n,n+m}$ holds.
Similar to $\Hc{2}$ control, the data is informative if there exist matrices $Y\succ 0$ and $L$ such that \eqref{opt:hinf_yl} holds and that
\begin{align}\label{opt:inclusion_hinf}
    \Zab{n}{n+m}(N)\subseteq\Zab{n}{n+m}^+(M_\Hc{\infty}).
\end{align}
Also, the converse holds when $\Sigma=\SigmaR$.
We can obtain the LMI conditions by applying Theorem~\ref{thm:Slem_new}.
\begin{theorem}\label{thm:hinf}
    Suppose that
    the data $(X_+,X,U)$ is generated by \eqref{eq:cleandata} and \eqref{eq:measured_data} under Assumption~\ref{ass:noiseset}.
    Then, the data $(X_+,X,U)$ is informative for $\mathcal{H}_{\infty}$ control with performance $\gamma>0$ if there exist a positive definite $Y\inX{S}{n}$, a matrix $L\inX{R}{m\times n}$, and scalars $\alpha\geq 0$ and $\beta>0$ such that
    \begin{subequations}\label{LMI_Hinf}
    \begin{align}
      & 
      \begin{bmatrix}
        Y\!-\!I_n\!-\!\beta I_n&0&0     &0   &0        \\
        0              &0&0     &Y   &0        \\
        0              &0&0     &L   &0        \\
        0              &Y&L^\top&Y   &\Cyl^\top\\
        0              &0&0     &\Cyl&\gamma^2 I_p
      \end{bmatrix}\!
      -\alpha
      \begin{bmatrix}
        N&{0}\\{0}&{0}
      \end{bmatrix}\succeq 0,\label{hinf:lmi1}\\
      &\sqmat{Y}{\Cyl^\top}{\Cyl}{\gamma^2 I_p}\succ 0,\label{hinf:lmi2}
    \end{align}
    \end{subequations}
    where $N$ is defined as \eqref{dfn_N}.
    In particular, when the perturbation condition \eqref{perterbation_condition} holds,
    \eqref{LMI_Hinf} is a necessary and sufficient condition.
    Moreover, if \eqref{LMI_Hinf} is feasible, the controller $K=LY^{-1}\inX{R}{m\times n}$ ensures that $\Hc{\infty}$ norm of $G(z)$ is less than $\gamma$ for any system in $\Sigma$.
\end{theorem}
\begin{proof}
    The proof is similar to that of Theorem~\ref{thm:h2}.
    We only prove that $Y\succ 0$ and $L$ satisfy \eqref{opt:hinf_yl} and \eqref{opt:inclusion_hinf} if and only if they satisfy \eqref{LMI_Hinf} with some $\alpha\geq 0$ and $\beta>0$.
    First, \eqref{opt:hinf_yl} is same as \eqref{hinf:lmi2}.
    Next, by applying Theorem~\ref{thm:Slem_new}, $Y\succ 0$ and $L$ satisfy \eqref{opt:inclusion_hinf} if and only if there exist $\alpha\geq 0$ and $\beta>0$ such that $M_\Hc{\infty}-\alpha N\succeq \blkdiag(\beta I_n,0_{n+m})$.
    This inequality can be equivalently converted into \eqref{hinf:lmi1} by applying Schur complement twice.
\end{proof}

\section{Data Informativity with Output Feedback Control under Data Perturbation}\label{sec:inout}
In this section, we extend the data informativity problem to the case of output feedback controllers.
Our analysis extends previous results~\cite{DDCQMI:Steentjes2022_Cont_Sys_Let_Covariance,DDCQMI:Waarde2024_TAC_AR,DDCQMI:Lidong2026_automatica}
for data perturbation
by eliminating the restrictive full-rank data assumption.

\subsection{Problem Formulation}
Consider an autoregressive (AR) model
\begin{equation}\label{AR:nominal}
 \textstyle{
  y(t)=\sum_{l=1}^{L}A_{l}^*y(t-l)+\sum_{l=0}^{L}B_{l}^*u(t-l)
  }
\end{equation}
with output $y(t)\inX{R}{p}$ and input $u(t)\inX{R}{m}$.
Assume that the system matrices $A_l^*\inX{R}{p\times p},l = 1,\dots,L$, $B_l^*\inX{R}{p\times m}, l=0,\dots,L$ are unknown.
By defining state $x(t)\inX{R}{(p+m)L}$ as
\begin{equation*}
  x(t)\coloneqq [y(t-1)^\top\cdots y(t-L)^\top\ u(t-1)^\top\cdots u(t-L)^\top]^\top,
\end{equation*}
we obtain a state space representation
\begin{equation}\label{AR:stateeq}
\left\{
\begin{array}{ll}
  x_+ & =\mathbf{A}^*x+\mathbf{B}^*u,\\
  y & =A^*x+B^*u
\end{array}
\right.
\end{equation}
with
\begin{equation}\label{AR:state_mat}
  \mathbf{A}^*\coloneqq \left[\begin{array}{cc|cc}
    \multicolumn{4}{c}{A^*}\\\hline
    I&0&0&0\\
    0&0&0&0\\
    0&0&I&0     
  \end{array}\right],\ 
  \mathbf{B}^*\coloneqq \left[\begin{array}{c}
    B^*\\\hline 0\\ I\\ 0
  \end{array}\right],
\end{equation}
and
  $A^*\coloneqq [A_1^*\ \cdots\ A_L^*\ B_1^*\ \cdots\ B_L^*],\ B^*\coloneqq B_0^*$,
where the time index $t$ is omitted.
As in~\cite{DDCQMI:Steentjes2022_Cont_Sys_Let_Covariance,DDCQMI:Waarde2024_TAC_AR,DDCQMI:Lidong2026_automatica}, we consider a strictly proper dynamic output feedback controller whose input depends on the $L$-long previous output time series.
This controller can equivalently be represented by a state feedback controller $u=Kx$ for~\eqref{AR:stateeq}~\cite{DDCQMI:Steentjes2022_Cont_Sys_Let_Covariance}.

Suppose that we have data $(Y,X,U)$ additively perturbed by $(\Delta_Y,\Delta_X,\Delta_U)$ where the clean data satisfies the dynamics, which is equivalent to
\begin{equation}\label{AR:outputeq_mat}
  [I\ A^*\ B^*]\mathbf{X}=[I\ A^*\ B^*]\Delta,
\end{equation}
where $\mathbf{X}=[Y^\top\ -X^\top\ -U^\top]^\top$ and $\Delta\coloneqq [\Delta_Y^\top\ -\Delta_X^\top\ -\Delta_Z^\top]^\top$.
We refer to $\Delta$ as data perturbation for the AR model.
As in Assumption~\ref{ass:noiseset}, we assume QMI and subspace constraints for data perturbation.
\begin{assumption}\label{ass:noiseset_AR}
 Data perturbation for the AR model satisfies
    $\Delta\in\mathcal{D}\coloneqq \{E\hDel|\hDel^\top\in\Zab{\hat{n}}{T}(\hPhi)\}$
  with known matrices $E\inX{R}{(p+m)(L+1)\times \hat{n}}$ and $\hPhi\in\mathbf{\Pi}_{\hat{n},T}$.
\end{assumption}

Data perturbation can describe exogenous disturbances and measurement noise for the AR model as well.
Indeed, $E=[I_p\ 0_{p,(p+m)L+m}]^\top$ and $E=I$ lead to the problems in~\cite{DDCQMI:Steentjes2022_Cont_Sys_Let_Covariance,DDCQMI:Waarde2024_TAC_AR} and that in~\cite{DDCQMI:BISOFFI2021}, respectively.

We define the set of systems for the AR model consistent with the data as
\begin{equation}\label{AR:dfn_sigma}
  \Sigma \coloneqq \left\{(A,B)\middle|\exists \Delta\in\mathcal{D}\st[I\ A\ B]\mathbf{X}=[I\ A\ B]\Delta\text{ holds}\right\}.
\end{equation}
Note that $(A,B)$ are not the system matrices of the state space equation~\eqref{AR:stateeq} but its unknown parameters.
We adopt this representation because the system matrices $(\mathbf{A},\mathbf{B})$ contain the constant submatrices, making it difficult to characterize the set of $(\mathbf{A},\mathbf{B})$ via a QMI.
Since $\Sigma$ is described by these unknown parameters rather than the entire system matrices, it is not straightforward to extend the discussion in Sec.~\ref{sec:quadratic_stabilization} to the AR model case.

Data informativity for quadratic stabilization for the AR model is defined as follows.
\begin{definition}\label{dfn:datainfo_AR}
  The data $(Y,X,U)$ is informative for quadratic stabilization for the AR model under data perturbation if there exist a positive definite $P\inX{S}{(p+m)L}$ and a controller $K\inX{R}{m\times (p+m)L}$ such that the Lyapunov inequality
  \begin{equation}\label{AR:qstab}
    P-(\mathbf{A}+\mathbf{B}K)P(\mathbf{A}+\mathbf{B}K)^\top\succ 0
  \end{equation}
  holds for any $(A,B)\in\Sigma$.
\end{definition}

\emph{Remark:} Although we focus on the quadratic stabilization in this section, the proposed approach is readily extensible to optimal control problems.
The Lyapunov inequality~\eqref{AR:qstab} is equivalent to $[\mathbf{A}\ \mathbf{B}]^\top\in\Zab{n}{n+m}(M)$, where $n=(p+m)L$ and the matrix $M$ is given by \eqref{dfn_M_qstab}.
By replacing the matrix $M$ with $M_\Hc{2}$ or $M_\Hc{\infty}$ defined in Sec.~\ref{sec:optimal}, we can formulate $\Hc{2}$ or $\Hc{\infty}$ control problem, respectively.

\subsection{LMI Characterization of Data Informativity}
First, we obtain a QMI characterization of $\Sigma$ by following the approach in Sec.~\ref{subsec:qstab_sysset}.
\begin{proposition}\label{prop:AR_sys_set}
    Let Assumption~\ref{ass:noiseset_AR} hold and define
    \begin{equation}\label{AR:dfn_tildesigma}
      \SigmaR\coloneqq \left\{(A,B)\middle| [A\ B]^\top\in\Zab{p}{(p+m)L+m}(N)\right\},
    \end{equation}
    where $N\coloneqq [E\ \mathbf{X}]\hPhi[E\ \mathbf{X}]^\top$.
    Then, $\Sigma\subseteq\SigmaR$ holds.
    In particular, $\Sigma=\SigmaR$ holds if $\im E\supseteq [I_p\ 0_{(p+m)L+m}]^\top$ or $\hPhi_{22}\prec 0$.    
\end{proposition}
\begin{proof}
    This proposition corresponds to the combination of Lemma~\ref{lem:conservative_characterization} and Theorem~\ref{thm:sys_set}, and follows by applying the same reasoning.
    The proof is omitted.
\end{proof}

The set $\Sigma$ is characterized by the QMI involving the unknown parameters $(A,B)$, whereas the Lyapunov inequality \eqref{AR:qstab} is expressed in terms of the system matrices $(\mathbf{A},\mathbf{B})$.
This difference prevents the direct application of Theorem~\ref{thm:Slem_new}.
To address this issue, we reformulate \eqref{AR:qstab} into conditions described by a QMI involving $(A,B)$.
From \eqref{AR:state_mat}, it follows that $\mathbf{A}=[A^\top\ J_1^\top]^\top,\mathbf{B}=[B^\top\ J_2^\top]^\top$, where
\begin{equation*}
  J_1\coloneqq\left[\begin{array}{cccc}
    I_{p(L-1)}&0&0         &0\\
    0         &0&0         &0\\
    0         &0&I_{m(L-1)}&0     
  \end{array}\right],\ 
  J_2\coloneqq \left[\begin{array}{c}
    0\\ I_m\\ 0
  \end{array}\right].
\end{equation*}
Substituting this into \eqref{AR:qstab} and applying Schur complement, we obtain the following lemma.
\begin{lemma}\label{lem:AR_QMI}
    Let $J_K$ be $J_1+J_2K$ and decompose $P\inX{S}{(p+m)L}$ as
    \begin{equation}\label{AR:P_decomposition}
        P=\sqmat{P_{11}}{P_{12}}{P_{21}}{P_{22}},\ P_{11}\inX{S}{p},\ P_{22}\inX{S}{p(L-1)+mL}.
    \end{equation}
    A pair of $P\succ 0$ and $K\inX{R}{m\times (p+m)L}$ satisfies \eqref{AR:qstab} if and only if it satisfies
    \begin{subequations}
        \begin{align}
          \label{AR:lem1}
          &[I_p\ A\ B]M_{\rm AR}[I_p\ A\ B]^\top,\\
          \label{AR:lem2}
          &Z\succ 0,
        \end{align}
    \end{subequations}
    where $Z$ and $M_{\rm AR}$ are defined as
    $Z\coloneqq P_{22}-J_KPJ_K^\top$ and
    \begin{equation}
    \begin{aligned}
      M_{\rm AR}\coloneqq&\blkdiag\left(P_{11},\,-\matvec{I}{K}P\matvec{I}{K}^\top\right)\\
      &-\matvec{P_{12}}{-\matvec{I}{K}PJ_K^\top}Z^{-1}\matvec{P_{12}}{-\matvec{I}{K}PJ_K^\top}^\top.
    \end{aligned}
    \end{equation}
\end{lemma}

\begin{proof}
    The closed-loop matrix is described as $\mathbf{A}+\mathbf{B}K = [A_{\rm cl}^\top\ J_K^\top]^\top$ where $A_{\rm cl} \coloneqq A+BK$.
    Substituting this into \eqref{AR:qstab}, we obtain
    \begin{align}\label{AR:prf_block}
      \sqmat{P_{11}}{P_{12}}{P_{21}}{P_{22}}
      -\sqmat{A_{\rm cl}PA_{\rm cl}^\top}{A_{\rm cl}PJ_K^\top}{J_KPA_{\rm cl}^\top}{J_KPJ_K^\top}\succ 0.
    \end{align}
    Applying Schur complement, we obtain
    $P_{11}-A_{\rm cl}PA_{\rm cl}^\top-(P_{12}-A_{\rm cl}PJ_K^\top)Z^{-1}(P_{21}-J_KA_{\rm cl}^\top)\succ 0$ and $Z\succ 0.$
    Since $A_{\rm cl} =[A\ B][I\ K^\top]^\top$, the first matrix inequality is equivalent to \eqref{AR:lem1}.
\end{proof}

From Proposition~\ref{prop:AR_sys_set} and Lemma~\ref{lem:AR_QMI}, it follows that matrices $P\succ 0$ and $K$ satisfy \eqref{AR:qstab} for any $(A,B)\in\SigmaR$ if and only if they satisfy \eqref{AR:lem2} and
\begin{equation}\label{AR:inclusion}
    \Zab{p}{(p+m)L+m}(N)\subseteq\Zab{p}{(p+m)L+m}^+(M_{\rm AR}).
\end{equation}
Therefore, since $\Sigma\subseteq\SigmaR$, data is informative if \eqref{AR:lem2} and \eqref{AR:inclusion} hold with some $P\succ 0$ and $K$.
Also, the condition becomes necessary for data informativity when $\Sigma=\SigmaR$.
The condition \eqref{AR:inclusion} can be converted into an LMI condition by Theorem~\ref{thm:Slem_new} and Schur complement.
Note that proposition~\ref{prop:Slem_beta} is inapplicable because $N\in\mathbf{\Pi}_{p,{(p+m)L+m}}$ does not necessarily hold.
\begin{theorem}\label{thm:LMI_AR}
    Suppose that
    the data $(Y,X,U)$ satisfying \eqref{AR:outputeq_mat} under Assumption~\ref{ass:noiseset_AR}.
    Then, the data $(Y,X,U)$ is informative for quadratic stabilization for the AR model if there exists 
    a positive definite $P\inX{S}{(p+m)L}$
    a matrix $L\inX{R}{m\times (p+m)L}$, and scalars $\alpha\geq 0$ and $\beta>0$ such that
    \begin{subequations}\label{AR:LMI_result}
    \begin{align}
      &\left[
        \begin{array}{ccc|cc}
          P_{11}-\beta I&0&0&P_{12}&0\\
          0&0&0&0&P\\
          0&0&0&0&L\\\hline
          P_{21}&0&0&P_{22}&J_{P,L}\\
          0&P&L^\top&J_{P,L}^\top&P
        \end{array}
      \right]-\alpha\left[\begin{array}{c|c}
        {N}&{0}\\\hline{0}&{0}
      \end{array}\right]\succeq 0,\label{AR:lmi1}\\
      &\sqmat{P_{22}}{J_{P,L}}{J_{P,L}^\top}{P}\succ 0,\label{AR:lmi2}
    \end{align}
    \end{subequations}
    where $J_{P,L}\coloneqq J_1P+J_2L$, $N\coloneqq [E\ \mathbf{X}]\hPhi[E\ \mathbf{X}]^\top$ and $P$ is decomposed as in \eqref{AR:P_decomposition}.
    In particular, when $\im E\supseteq\im[I_p\ 0_{p,(p+m)L+m}]^\top$ or $\hPhi_{22}\prec 0$, \eqref{AR:LMI_result} is a necessary and sufficient condition.
    Moreover, if \eqref{AR:LMI_result} is feasible, the controller $K=LP^{-1}$ stabilizes any system in $\Sigma$.
\end{theorem}
\begin{proof}
    First, we confirm that $M_{\rm AR}$ and $N$ satisfy the preconditions of Theorem~\ref{thm:Slem_new} under $P\succ 0$ and \eqref{AR:lem2}.
    From the definition, the submatrices of $M_{\rm AR}$ are
    \begin{align}
      M_{{\rm AR},{22}}&=-\matvec{I}{K}(P+PJ_K^\top Z^{-1}J_KP)\matvec{I}{K}^\top,\\
      M_{{\rm AR},{12}}&=-P_{12}Z^{-1}J_KP\matvec{I}{K}^\top.
    \end{align}
    Clearly, $M_{{\rm AR},{22}}\preceq0$ holds.
    Also, since 
    $\ker M_{{\rm AR},22} =\ker [I\ K^\top]$, it follows that $\ker M_{{\rm AR},22}\subseteq\ker M_{{\rm AR},12}$.
    Moreover, $\Zab{p}{(p+m)L+m}(N)$ is nonempty because $\SigmaR$ contains the true system.

    Under $P\succ 0$ and \eqref{AR:lem2}, it follows from Theorem~\ref{thm:Slem_new} that
    $P$ and $K$ satisfy \eqref{AR:inclusion} if and only if there exist $\alpha\geq 0$ and $\beta>0$ satisfying $M_{\rm AR}-\alpha N\succeq \blkdiag(\beta I,0_{(p+m)L+m})$, which can be equivalently converted into \eqref{AR:lmi1} by introducing $L=KP$ and applying the Schur complement.
    Moreover, under $P\succ 0$, \eqref{AR:lem2} is equivalent to \eqref{AR:lmi2}.
    Therefore, if and only if there exists $P\succ 0,L,\alpha\geq 0$ and $\beta>0$ such that \eqref{AR:LMI_result} holds, $P$ and $K=LP^{-1}$ satisfy \eqref{AR:lem2} and \eqref{AR:inclusion}, which is sufficient for data informativity and also necessary if $\SigmaR=\Sigma$.
    Note that $\SigmaR=\Sigma$ follows from Proposition~\ref{prop:AR_sys_set} when $\im E\supseteq\im[I_p\ 0_{p,(p+m)L+m}]^\top$ or $\hPhi_{22}\prec 0$.
\end{proof}

\emph{Remark:}
We compare our approach in this section with the previous studies that address exogenous disturbances~\cite{DDCQMI:Steentjes2022_Cont_Sys_Let_Covariance, DDCQMI:Waarde2024_TAC_AR}.
In the previous studies, the matrix S-procedure is applied to QMIs involving $H_1[A\ B]$ with $H_1\coloneqq [I_p\ 0_{p,(p+m)L-p}]^\top$. This formulation is used because the Lyapunov inequality can be readily transformed into such a QMI using the relation
$[\mathbf{A}\ \mathbf{B}] = H_1[A\ B] + H_2[J_1\ J_2],$
where $H_2\coloneqq [0_{(p+m)L-p,p}\ I_{(p+m)L-p}]^\top$.
Also, the QMI characterizing $\Sigma$ with $N$ can be rewritten as a QMI involving $H_1[A\ B]$ by applying Proposition~\ref{prop:lin_eq_qmi} in Appendix.
Then, LMI conditions for data informativity are derived by applying Proposition~\ref{prop:Slem_beta} to these two QMIs  involving $H_1[A\ B]$.
However, a nonconservative application of Proposition~\ref{prop:lin_eq_qmi} requires the condition $N_{22}\prec 0$. 
Since $N_{22}=[X^\top\ U^\top]^\top\Phi_{22}[X^\top\ U^\top]$ in the exogenous disturbance case, this condition demands that the data matrix $[X^\top\ U^\top]^\top$ has full row rank.
Indeed, this full-rank property appears in Lemma 1 of~\cite{DDCQMI:Steentjes2022_Cont_Sys_Let_Covariance} and in Assumption 4 and Theorem 20 of~\cite{DDCQMI:Waarde2024_TAC_AR}.
In contrast, our approach directly uses QMIs involving $[A\ B]$ and therefore does not rely on Proposition~\ref{prop:lin_eq_qmi}, enabling us to handle rank-deficient data.

\section{Data Informativity under Structured Data Perturbation}\label{sec:structured}
This section introduces structured data perturbation.
We derive a sufficient condition for data informativity under the structured data perturbation.

\subsection{Representation of Structured Data Perturbation}
We introduce the structured data perturbation given by the sum of linearly transformed perturbations each of which is characterized by a single QMI.
\begin{assumption}\label{ass:noiseset_sum}
    Data perturbation satisfies
    \begin{equation}\label{sum:noise_sum}
        \Delta\in\mathcal{D}_{\rm str}\coloneqq\left\{
        \textstyle{\sum_{j=1}^{J} E_j\Delta_jF_j}\middle|\textstyle{\Delta_j^\top\in\Zab{n_j}{T_j}(\Phi_j)}
        \right\}
    \end{equation}
    with known matrices $E_j\inX{R}{n_d\times n_j}, F_j\inX{R}{T_j\times T}$ and $\hPhi_j\in\mathbf{\Pi}_{n_j,T_j}$ for $j = 1,\dots,J$, where $n_d$ denotes an appropriate dimension selected to match the data.
\end{assumption}

Assumption~\ref{ass:noiseset_sum} accommodates a variety of perturbation structures through the following representative scenarios.
\begin{enumerate}
    \item \label{sum:enume_superposition}
    Superposition of exogenous disturbance and measurement noise, which results in data $(X_+,X,U)$ perturbed as $X_+=X_+^*+\Delta_{X_+}$, $X=X^*+\Delta_X$, $U=U^*+\Delta_U$ and $X_{+}^*=A^*X^*+B^*U^*+D$.
    We consider the case where the prior knowledge on the disturbance $D$ and the measurement noise $(\Delta_{X_+},\Delta_X,\Delta_U)$ is obtained separately:
    The disturbance $D$ and the measurement noise $\Delta_m= [\Delta_{X_+}^\top\,-\Delta_X^\top\,-\Delta_U^\top]^\top$ are assumed to satisfy $D^\top\in\Zab{n}{T}(\Phi_d)$ and $\Delta_m^\top\in\Zab{2n+m}{T}(\Phi_m)$.
    The combined perturbation set can then be represented as a summation:
    \begin{equation}
        \mathcal{D}_{\rm str} = \left\{E_d D+\Delta_m\middle|
        \begin{aligned}
        &D^\top\in\Zab{n}{T}(\Phi_d),\\
        &\Delta_m^\top\in\Zab{2n+m}{T}(\Phi_m)\end{aligned} \right\},
    \end{equation}
    where $E_d=[I_n\ 0_{n,n+m}]^\top$.
    
    \item Hankel-structured perturbation
    \begin{align}
    &\Delta=\begin{bmatrix}
      \delta_0(0)&\cdots&\delta_0(T-1)\\
      \vdots&\ddots&\vdots\\
      \delta_0(L-1)&\cdots&\delta_0(T+L-2)
    \end{bmatrix},
    \end{align}
    where $\Delta_0=[\delta_0(0)\cdots\delta_0(T+L-2)]\inX{R}{p\times (T+L-1)}$ is subject to $\Delta_0^\top\in\Zab{d}{T+L-1}(\Phi_0)$ with a known matrix $\Phi_0$.
    The corresponding perturbation set can be expressed with $\Delta = \sum_{l=0}^{L-1}E_l\Delta_0F_l$ with $E_l = [0_{p,pl}\ I_p\ 0_{p,p(L-l-1)}]^\top$ and $F_l = [0_{T,l}\ I_T\ 0_{T,L-l-1}]^\top$.
    This form of structured perturbation typically arises in the sequential data.

    \item Instantaneously bounded perturbation $\Delta = E[\delta(0)\cdots\delta(T-1)]$ subject to $\delta(t)^\top\in\Zab{d}{1}(\Phi_t)$ at each time step $t=0,\dots,T-1$, which has been studied in~\cite{DDCQMI:BISOFFI2021,DDCQMI:BISOFFI2024_CSL}.
    The total perturbation can be written in the summation form with $\Delta = \sum_{t=0}^{T-1}E\Delta_tF_t$ holds where $F_t=\univec{t+1}{T}^\top$.

    \item \label{sum:enume_elementwise}
    Element-wise bounded perturbation where each $(i,j)$-th entry $\delta_{ij}$ is individually bounded as $\delta_{ij}\in\Zab{1}{1}(\Phi_{ij})$.
    The perturbation set can then be represented with $\Delta =\sum_{i,j}E_i\delta_{ij}F_j$ with $E_i = \univec{i}{n_d}$ and $F_j = \univec{j}{T}^\top$.
    
    \item 
    A combination of the aforementioned scenarios is also possible. For instance, \cite{DDCQMI:Hu2025_RDPC} considers a case where data involve both exogenous disturbances and measurement noise with element-wise bounds. These constraints correspond to a combination of \ref{sum:enume_superposition}) and \ref{sum:enume_elementwise}) in our formulation.
\end{enumerate}

We approximate $\mathcal{D}_{\rm str}$ with the solution set of a single QMI, thereby reducing the informativity problem under Assumption~\ref{ass:noiseset_sum} to that in the previous sections.

\subsection{Outer QMI Approximation of Structured Perturbation Set}
Consider the set $\bar{\mathcal{D}}(\Phi)\coloneqq \{\Delta|\Delta^\top\in\Zab{n_d}{T}(\Phi)\}$ characterized by a single QMI parameterized by a matrix $\Phi$.
First, we derive an LMI characterization for $\Phi$ such that the set serves as an outer approximation of the structured perturbation set $\mathcal{D}_{\rm str}$, i.e., $\mathcal{D}_{\rm str}\subseteq \bar{\mathcal{D}}(\Phi)$, based on an ellipsoidal approximation technique for sums of vector-valued ellipsoids, as detailed in~\cite[Sec. 3.7.4]{boyd1994linear}.

We introduce the horizontally concatenated matrix $\Gamma\coloneqq [E_1\Delta_1\ E_2\Delta_2\ \cdots\ E_J\Delta_J]\inX{R}{n_d\times \sum_{j=1}^{J}T_j}$.
We characterize the condition $\mathcal{D}_{\rm str}\subseteq \bar{\mathcal{D}}(\Phi)$ using QMIs.
\begin{lemma}\label{lem:stack}
    Let Assumption~\ref{ass:noiseset_sum} hold and denote $\mathbf{T} = \sum_{j=1}^JT_j$.
    The approximation
    $\mathcal{D}_{\rm str}\subseteq \bar{\mathcal{D}}(\Phi)$ is established if 
    \begin{equation}\label{sum:lem_stack}
        \textstyle{
        \bigcap_{1\leq j\leq J}\Zab{n_d}{\mathbf{T}}(\Psi_j)\subseteq\Zab{n_d}{\mathbf{T}}(\Psi),
        }
    \end{equation}
    where
    {\small\setlength{\arraycolsep}{2pt}\begin{align*}
        &\Psi_j \coloneqq \sqmat{E_j}{0}{0}{U_j}\Phi_j\sqmat{E_j}{0}{0}{U_j}^\top\!,\Psi \coloneqq \sqmat{I_{n_d}}{0}{0}{U}\Phi\sqmat{I_{n_d}}{0}{0}{U}^\top\!,\\
        &U_j\coloneqq \left[0_{T_j,\sum_{k=1}^{j-1}T_k}\ \ I_{T_j}\ \ 0_{T_j,\sum_{k=j+1}^{J}T_k}\right]^\top,U\coloneqq \sum_{j=1}^{J}U_jF_j.
    \end{align*}}
\end{lemma}
\begin{proof}
    Assume that \eqref{sum:lem_stack} holds and $\Delta_1,\dots,\Delta_J$ satisfy $\Delta_j^\top\in\Zab{n_j}{T_j}(\Phi_j)$ for each $j$.
    To prove $\mathcal{D}_{\rm str}\subseteq \bar{\mathcal{D}}(\Phi)$, we will show that $\Delta\coloneqq\sum_{j=1}^{J}E_j\Delta_j F_j$ satisfies $\Delta\in\bar{\mathcal{D}}(\Phi)$.
    It follows from Proposition~\ref{prop:lin_eq_qmi} in Appendix that $(E_j\Delta_j)^\top\in\Zab{n_d}{T_j}(\bar{\Phi}_j)$ holds for each $j$, where $\bar{\Phi}_j\coloneqq \blkdiag(E_j,I_{T_j})\Phi_j\blkdiag(E_j^\top,I_{T_j})$.
    Since $\Gamma U_j =E_j\Delta_j$ follows from the definition of $\Gamma$ and $U_j$, we obtain $(\Gamma U_j)^\top\in\Zab{n}{T_j}(\bar{\Phi}_j)$, which can be rewritten as
    \begin{equation}
        [I\ \ \Gamma U_j]\bar{\Phi}_j[I\ \ \Gamma U_j]^\top=[I\ \ \Gamma]\Psi_j[I\ \ \Gamma]^\top\succeq 0.
    \end{equation}
    Then, we have
        $\Gamma^\top\in\Zab{n_d}{\mathbf{T}}(\Psi_j)$.
    Since this equation is satisfied for every $j$, it follows that
        $\Gamma^\top\in\bigcap_{1\leq j\leq J}\Zab{n_d}{\mathbf{T}}(\Psi_j)$.
    From~\eqref{sum:lem_stack}, 
    it follows that $\Gamma^\top\in\Zab{n_d}{\mathbf{T}}(\Psi)$.
    By using $\Gamma U = \Delta$, we can confirm
        $[I\ \ \Gamma]\Psi[I\ \ \Gamma]^\top=[I\ \ \Delta]\Phi[I\ \ \Delta]^\top\succeq 0.$
    This inequality means that $\Delta\in\bar{\mathcal{D}}(\Phi)$.
\end{proof}

We obtain the following theorem by applying the (conservative) matrix S-procedure to multiple QMIs (Proposition~\ref{prop:Slem_lossy} in Appendix).
\begin{theorem}\label{thm:approximation}
    Let Assumption~\ref{ass:noiseset_sum} hold.
    The condition $\mathcal{D}_{\rm str}\subseteq \bar{\mathcal{D}}(\Phi)$ holds if there exist $\alpha_j\geq 0$, $j=1,\dots,J$ such that
    {\setlength{\arraycolsep}{5pt}\begin{equation}\label{sum:lmi_Phi}
        \sqmat{I}{0}{0}{U}\Phi\sqmat{I}{0}{0}{U}^\top\! -\sum_{j=1}^{J}\alpha_j \sqmat{E_j}{0}{0}{U_j}\Phi_j\sqmat{E_j}{0}{0}{U_j}^\top\!\succeq 0.
    \end{equation}}
\end{theorem}
\begin{proof}
    By using the matrices defined in Lemma~\ref{lem:stack}, \eqref{sum:lmi_Phi} is equivalent to
    $\Psi-\sum_{j=1}^{J}\alpha_j\Psi_j\succeq 0.$
    By applying Proposition~\ref{prop:Slem_lossy}, it follows that \eqref{sum:lem_stack}.
    Therefore, we obtain $\mathcal{D}_{\rm str}\subseteq \bar{\mathcal{D}}(\Phi)$ from Lemma~\ref{lem:stack}.
\end{proof}

Theorem~\ref{thm:approximation} establishes an LMI condition for constructing an outer approximation of the structured perturbation set $\mathcal{D}_{\rm str}$.
Note that the feasible set of~\eqref{sum:lmi_Phi} includes a trivial approximation.
In particular, choosing $\Phi\succeq 0$ and $\alpha_j=0,\ j=1,\dots,J$ satisfies the LMI but results in $\bar{\mathcal{D}}(\Phi)$ encompassing the entire space, thereby rendering the approximation vacuous.
Therefore, to obtain a meaningful and informative outer approximation, one must carefully select nontrivial feasible solutions to the LMI.

\subsection{Co-design of Outer QMI Approximation and Stabilizing Controller}
We derive a sufficient condition of data informativity for the quadratic stabilization via state feedback under the structured perturbation.
Let $n_d=2n+m$ and data $\mathbf{X}=[X_+^\top\ -X^\top\ -U^\top]^\top$ compromised by data perturbation $\Delta\in\mathcal{D}_{\rm str}$ satisfy
\begin{equation}\label{sum:stateeq}
  [I\ A\ B]\mathbf{X}=[I\ A\ B]\Delta.
\end{equation}
Then, the set of systems consistent with the data is given by
$\Sigma=\left\{(A,B)\middle|\exists \Delta\in\mathcal{D}_{\rm str}\st \eqref{sum:stateeq}\text{ holds}\right\}.$
Take a matrix $\Phi\inX{S}{(2n+m)+T}$ feasible in~\eqref{sum:lmi_Phi}.
Since $\mathcal{D}_{\rm str}\subseteq \bar{\mathcal{D}}(\Phi)$ from Theorem~\ref{thm:approximation}, we have
$\Sigma\subseteq \Sigma(\Phi)\coloneqq \left\{(A,B)\middle|\Delta\in\bar{\mathcal{D}}(\Phi)\st \eqref{sum:stateeq}\text{ holds}\right\}.$
Therefore, quadratic stabilization of $\Sigma(\Phi)$ is sufficient for that of $\Sigma$.
The condition for the quadratic stabilization of $\Sigma(\Phi)$ can be derived in the same manner as in Sec.~\ref{sec:quadratic_stabilization}.
First, we derive the QMI characterization of $\Sigma(\Phi)$.

\begin{lemma}
    Define $\bar{\Sigma}(\Phi)$ as
    \begin{equation}
        \bar{\Sigma}(\Phi)\coloneqq \left\{(A,B)\middle|[A\ B]^\top\in\Zab{n}{n+m}(N(\Phi))\right\},
    \end{equation}
    where
    $N(\Phi)\coloneqq[I\ \mathbf{X}]\Phi[I\ \mathbf{X}]^\top.$
    Then $\Sigma(\Phi)\subseteq\bar{\Sigma}(\Phi)$ holds.
\end{lemma}
\begin{proof}
    The proof is omitted as it is analogous to that of Lemma~\ref{lem:conservative_characterization}.
\end{proof}

Since the system set $\bar{\Sigma}(\Phi)$ is characterized by a QMI, we can apply Theorem~\ref{thm:Slem_new}.
After Schur complement with change of variables, we obtain the following sufficient condition for quadratic stabilization:
\begin{equation}\label{sum:lmi_qstab}
  \mathcal{M}(P,L,\beta)-\alpha\,\blkdiag(N(\Phi),0_n)\succeq 0,
\end{equation}
where $K=LP^{-1}$ is a stabilizing controller.
The remaining problem is the selection of a suitable matrix $\Phi$ as the outer QMI approximation from the feasible solutions of \eqref{sum:lmi_Phi}.

We propose a co-design approach for simultaneously determining the outer approximation matrix $\Phi$ and a stabilizing controller $K$.
This co-design problem is formulated as follows.
\begin{equation}\label{sum:BMI_dinfo}
  \begin{split}
    \text{find }&\Phi,\alpha_1,\dots,\alpha_J,P,L,\alpha,\beta,\\
    \st&\alpha_j\geq 0,P\succ 0,\alpha\geq 0,\beta>0,\eqref{sum:lmi_Phi},\eqref{sum:lmi_qstab}.
  \end{split}
\end{equation}
The problem~\eqref{sum:BMI_dinfo} has a bilinear matrix inequality (BMI) structure, as it contains the product of the decision variables $\alpha$ and $\Phi$ in~\eqref{sum:lmi_qstab}.
This bilinearity renders the problem NP-hard in general.
A commonly adopted heuristic is a two-step procedure~\cite{Con:shimomura2005multiobjective,Con:Tran2012} by decoupling the problem by first selecting a fixed matrix $\Phi$ that satisfies the constraint~\eqref{sum:lmi_Phi}, and subsequently solving~\eqref{sum:lmi_qstab} with this fixed $\Phi$. However, it is important to note that this heuristic does not guarantee global optimality of the resulting solution.

Remarkably, the optimization problem~\eqref{sum:BMI_dinfo}, despite its inherent bilinearity, can be equivalently reformulated as an optimization problem involving only LMIs.
The key observation is that the LMI~\eqref{sum:lmi_Phi} is homogeneous, which allows for normalization by setting $\alpha = 1$ without loss of generality.
\begin{theorem}\label{thm:alpha_scale}
  The feasibility of the problem \eqref{sum:BMI_dinfo} is equivalent to the feasibility of
  \begin{equation}\label{sum:LMI_dinfo}
    \begin{split}
      \text{find }&\Phi,\alpha_1,\dots,\alpha_J,P,L,\beta,\\
      \st&\alpha_j\geq 0,P\succ 0,\beta>0,\eqref{sum:lmi_Phi},\\
      &\mathcal{M}(P,L,\beta)-\blkdiag(N(\Phi),0_n)\succeq 0.
    \end{split}
  \end{equation}
\end{theorem}
\begin{proof}
  If the problem \eqref{sum:LMI_dinfo} is feasible, we can obtain a feasible solution of \eqref{sum:BMI_dinfo} by adding $\alpha=1$ to the solution of \eqref{sum:LMI_dinfo}.
  Conversely, assume that \eqref{sum:BMI_dinfo} is feasible and that $(\Phi,\alpha_1,\dots,\alpha_J,P,L,\alpha,\beta)$ is a solution of the problem.
  Then, $(\alpha\Phi,\alpha\alpha_1,\dots,\alpha\alpha_J,P,L,\beta)$ is a feasible solution of \eqref{sum:LMI_dinfo} because $(\alpha\Phi,\alpha\alpha_1,\dots,\alpha\alpha_J)$ is a solution of \eqref{sum:lmi_Phi} due to its homogeneity.
  Note that $\alpha\alpha_j\geq 0$ holds from $\alpha\geq0$ and $\alpha_j\geq0$.
\end{proof}

Theorem~\ref{thm:alpha_scale} directly leads to a sufficient LMI condition for data informativity under structured data perturbations.
\begin{corollary}
    Suppose that the data $(X_+,X,U)$ and the data perturbation under Assumption~\ref{ass:noiseset} satisfy~\eqref{sum:stateeq}.
    Then, the data is informative for quadratic stabilization if \eqref{sum:LMI_dinfo} is feasible.
    Moreover, if \eqref{sum:LMI_dinfo} is feasible, the controller $K=LP^{-1}\inX{R}{m\times n}$ stabilizes any system in $\Sigma$.
\end{corollary}

By following similar procedures, we can also derive LMI-based sufficient conditions for data informativity regarding optimal control and quadratic stabilization of AR models under Assumption~\ref{ass:noiseset_sum}.
In Sec.~\ref{subsec:exp_codesign}, we demonstrate that our co-design approach outperforms the two-step heuristic approach.

\emph{Remark:}
Since~\eqref{sum:lmi_qstab} is also homogeneous, its feasibility remains unchanged when $\alpha$ is fixed to $\alpha=1$~\cite[Theorem 5.1]{DDCQMI:Waarde2023_siam_qmi}.
As a result, the LMI characterization~\eqref{sum:LMI_dinfo} for quadratic stabilization remains valid without explicitly exploiting the homogeneity of the outer approximation condition~\eqref{sum:lmi_Phi}, as performed in the proof of Theorem~\ref{thm:alpha_scale}.
However, LMI characterizations corresponding to other control objectives, such as optimal control, do not necessarily exhibit this homogeneity property.
Therefore, our equivalent transformation to LMIs is particularly valuable for addressing such general cases.

\emph{Remark:}
We compare the method proposed in this section with prior studies.
First,~\cite{DDCQMI:BISOFFI2021} addresses a case where exogenous disturbances are bounded instantaneously, which can be represented by Assumption~\ref{ass:noiseset_sum}.
Although our approach differs from~\cite{DDCQMI:BISOFFI2021}, both yield equivalent LMIs as demonstrated below.
The work~\cite{DDCQMI:BISOFFI2021} considers a system exposed to exogenous disturbance
\begin{equation}\label{sum:stateeq_dist}
    x(t+1)=Ax(t)+Bu(t)+d(t),
\end{equation}
where the exogenous disturbance signal at each time step satisfies $d(t)^\top\in\Zab{n}{1}(\Phi_{\rm inst})$ with {$\Phi_{\rm inst}=\blkdiag(\epsilon I_n,-1)$}.
The set of systems consistent with the $T$-long data $(X_+,X,U)$ is given by $\Sigma_{\rm inst} \coloneqq \cap_{t=0}^{T-1}\mathcal{C}_t$ with
\[
\mathcal{C}_t\!\coloneqq\! \{(A,B)|\exists d(t)\inX{R}{n}\st d(t)^\top\!\in\Zab{n}{1}(\Phi_{\rm inst}),\eqref{sum:stateeq_dist}\text{ hold}\}.
\]
Each set $\mathcal{C}_t$ corresponds to the solution set of a QMI involving system matrices.
Applying Proposition~\ref{prop:Slem_lossy} results in a sufficient condition for quadratic stabilization~\cite[Proposition 1]{DDCQMI:BISOFFI2021}.
On the other hand, our approach treats the $T$-long disturbance matrix as the data perturbation in the form of 
\begin{equation}
    \Delta\in\mathcal{D}_{\rm inst}\coloneqq \left\{E_{d}\textstyle{\sum_{t=0}^{T-1}}d(t)\univec{t+1}{T}^\top \middle| d(t)^\top\in\Zab{n}{1}(\Phi_{\rm inst})\right\},
\end{equation}
which satisfies Assumption~\ref{ass:noiseset_sum}.
Under these conditions, the LMI~\eqref{sum:lmi_Phi} in Theorem~\ref{thm:approximation} simplifies to
\begin{align}\label{sum:hikaku}
    \hspace{-10pt}\Phi -\sum_{t=0}^{T-1}\alpha_t \sqmat{E_{d}}{0}{0}{\univec{t+1}{T}}\Phi_{\rm inst}\sqmat{E_{d}}{0}{0}{\univec{t+1}{T}}^\top\succeq 0.
\end{align}
By substituting~\eqref{sum:hikaku} into~\eqref{sum:LMI_dinfo}, the variable $\Phi$ can be eliminated, resulting in a condition equivalent to the condition in~\cite[Proposition 1]{DDCQMI:BISOFFI2021}.
Therefore, our method can provide an alternative approach for the problem in~\cite{DDCQMI:BISOFFI2021}.
Also, our formulation accommodates broader classes of noise models beyond those considered in~\cite{DDCQMI:BISOFFI2021}, because the consistent system set is not generally described as $\Sigma_{\rm inst}=\cap_{t=0}^{T-1}\mathcal{C}_t$ under the general perturbation structure imposed by Assumption~\ref{ass:noiseset_sum}.

\emph{Remark:}
We compare our framework with the method proposed in~\cite{DDCQMI:Hu2025_RDPC}.
The work~\cite{DDCQMI:Hu2025_RDPC} addresses model predictive control by utilizing multiple input-state-output datasets which contain element-wise-bounded disturbances and measurement noises.
Although this control objective differs from ours, the noise representation in~\cite{DDCQMI:Hu2025_RDPC} is covered by Assumption 6.1.
A major distinction lies in the approximation method of a set $\bar{\mathcal{D}}(\Phi)$:
Whereas~\cite{DDCQMI:Hu2025_RDPC} employs a fixed approximation, our approach provides a flexible approximation that is optimized simultaneously when solving \eqref{sum:LMI_dinfo}.
Although these works address different problems, the idea of the flexible approximation has potential for reducing conservativeness in~\cite{DDCQMI:Hu2025_RDPC}, which presents an interesting avenue for future research.

\section{Numerical example}\label{sec:exam}
This section presents numerical examples to illustrate the theoretical findings.
In all simulations, the dataset is generated as follows: the elements of the data matrices $X$ and $U$ are independently sampled from the standard normal distribution.
The matrix $X_+$ is then generated using the true system dynamics and data perturbation.
Data perturbations are sampled from a uniform distribution over the prescribed perturbation set $\mathcal{D}$ or $\mathcal{D}_{\rm str}$, depending on the scenario.
We employ the Metropolis algorithm~\cite{robert2004monte} to sample perturbations from the uniform distribution over these nontrivial sets.

\subsection{Visualization of Quadratic Stabilization}\label{subsec:1d_qstab}
We visualize the quadratic stabilization via state feedback, discussed in Sec.~\ref{sec:quadratic_stabilization}, using a one-dimensional linear system.
We set the unknown true system as $(A^*,B^*)=(1.2,0.6)$ and collect data of length $T=20$.
We consider the case where the data is perturbed by measurement noise.
To model this scenario, we set the matrices in Assumption~\ref{ass:noiseset} as $E=I_{2n+m}$ and 
$\hPhi = \blkdiag(\epsilon^2TI_{2n+m}, -I_T)$
with $\epsilon=0.3$.
The resulting perturbation set is given by
\begin{equation}\label{exp:noiseset}
    \mathcal{D}=\left\{\Delta\middle|\Delta^\top\in\Zab{2n+m}{T}\left(\sqmat{\epsilon^2TI_{2n+m}}{0}{0}{-I_T}\right)\right\}.
\end{equation}
Then, the LMI \eqref{LMI_result} is feasible, and the resulting variables $P$ and $L$ lead to a stabilizing controller $K=LP^{-1}$.
Fig.~\ref{fig:1d_example} depicts the system set $\Sigma$ as a dotted ellipse and the set of systems stabilized by $K$ as a dashed band-shaped region.
We observe that all systems in $\Sigma$ lie within the stabilizing region, confirming that the controller $K$ stabilizes any system in $\Sigma$.

\begin{figure}[t]
  \centerline{\includegraphics[width=\columnwidth]{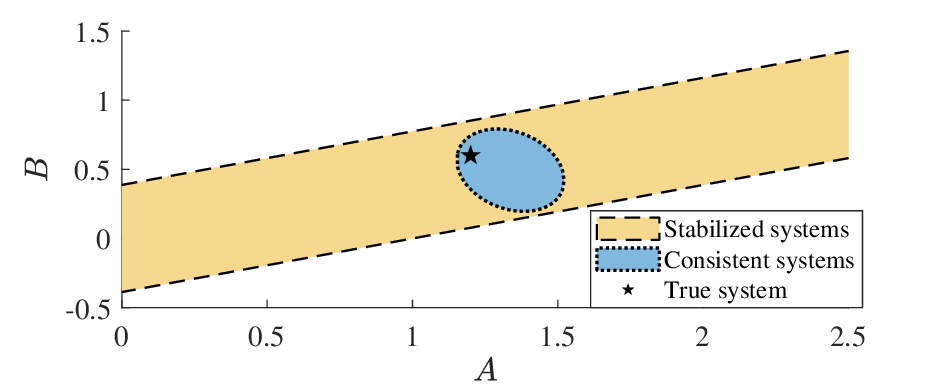}}
  \caption{Quadratic stabilization of $\Sigma$.}
  \label{fig:1d_example}
\end{figure}

\subsection{Example of Informative Rank-deficient Data}\label{subsec:muri}
In Sec.~\ref{subsec:rank-deficient}, we have discussed the informativity of rank-deficient data.
Here, we show an example of this scenario.
We use the model of a transport aircraft in a wings-level steady state flight condition~\cite{cntr:stevens2015aircraft}.
We discretize the model from~\cite[Example 3.7.1]{cntr:stevens2015aircraft} with a sampling time of $0.05$~s and
employ the resulting system as the unknown true system $A^*$ and $B^*$, given by
{\small\setlength{\arraycolsep}{1.5pt}\begin{equation}
    \begin{bmatrix*}[r]
         0.9986& 0.8315&-1.5535&-0.0178\\
        -0.0001& 0.9741&-0.0020& 0.0488\\
         0.0000&-0.0006& 1.0000& 0.0495\\
         0.0000&-0.0242& 0.0003& 0.9784
    \end{bmatrix*},\begin{bmatrix*}[r]
         2.3712   &-0.3130\\
        -0.0036   & 0.0008\\
        -0.0208   &0.0007\\
         0.0001   &-0.0000
    \end{bmatrix*},
\end{equation}} respectively.
We set the data length to $T=5$ so that the data $Z=[X^\top\ U^\top]^\top$ does not have full row rank, since $T<n+m$.
We consider the case where noise affects the state measurement, realized by setting the matrices in Assumption~\ref{ass:noiseset} to be $E=[I_{2n}\ 0_{2n,m}]^\top$ and 
$\hPhi = \blkdiag(\hat{\Theta},-I_T)$
with $\hat{\Theta}=0.005^2TI_{2n}$.
As discussed in Sec.~\ref{subsec:rank-deficient}, data informativity requires that $N_{22}=E_-\hat{\Theta}E_-^\top-ZZ^\top\preceq 0$.
To satisfy this condition, we impose a structure on the input data $U=[U_1^\top\ U_2^\top]^\top$, where $U_1,U_2\inX{R}{1\times T}$: We set $U_2$ to be the zero matrix, while sampling $U_1$ and $X$ from the standard normal distribution.
Then, the LMI~\eqref{LMI_result_cor} is feasible and we obtain the stabilizing controller
\begin{equation*}
  K=\begin{bmatrix}
    -0.3288   &-0.7433    &0.6937    &0.4205\\
          0   &      0    &     0    &     0
  \end{bmatrix}.
\end{equation*}
The singularity of $N_{22}$ results in the second row of $K$ being zero.
It can be confirmed that the image of $K=V_{-}YP^{-1}$ is included in the image of $V_{-}$ given as a submatrix of $V$, which is spanned by orthogonal bases of $\im N_{22}$.
This result coincides with the claims presented in Sec.~\ref{subsec:rank-deficient}.

\subsection{Variation of Control Performance with Data Length}\label{subsec:exp_Perf_Length}

We consider the $\mathcal{H}_2$ optimal control using 
the transport aircraft model employed in Sec.~\ref{subsec:muri}, with the performance output defined as $z(t)=[x(t)^\top\  10^{-5}u(t)^\top]^\top$.
Specifically, we focus on the energy-bounded measurement noise case~\eqref{exp:noiseset} with $\epsilon=0.01$.
We evaluate the effect of data length $T$ by varying it in the range $10 \leq T \leq 10^3$.
For each value of $T$, we solve the optimization problem~\eqref{opt:opt_h2} 500 times using independently sampled datasets and record the average $\mathcal{H}_2$ performance $\gamma = \sqrt{J}$ over the feasible instances.
Fig.~\ref{fig:h2_perf} shows that, as the data length $T$ increases, the $\mathcal{H}_2$ performance improves and approaches the nominal performance.
This trend is consistent with prior results established for systems affected by exogenous disturbances~\cite{DDCQMI:Waarde2022_TAC_origin,DDCQMI:Steentjes2022_Cont_Sys_Let_Covariance}.

\begin{figure}[t]
  \centerline{\includegraphics[width=\columnwidth]{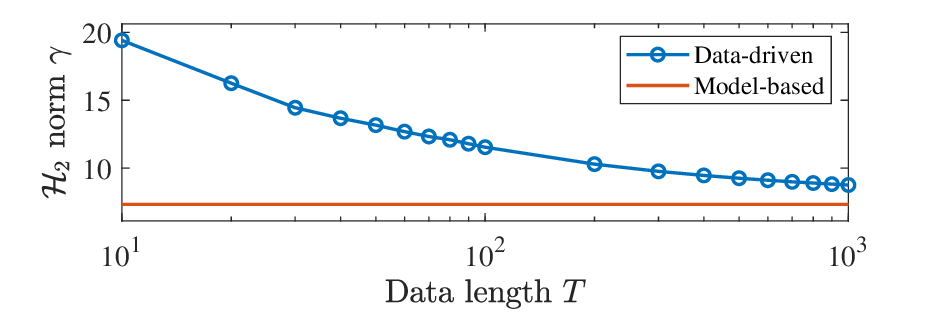}}
  \caption{$\mathcal{H}_2$ performance.}
  \label{fig:h2_perf}
\end{figure}

\subsection{Effectiveness of Co-design for Structured Perturbation}\label{subsec:exp_codesign}

We verify the effectiveness of the co-design method by solving the LMI problem \eqref{sum:LMI_dinfo} by comparing it with a heuristic approach to solving the BMI problem~\eqref{sum:BMI_dinfo}.
As a heuristic approach, we adopt a two-step method: first, we compute the matrix $\Phi_{\rm app}$ that satisfies \eqref{sum:lmi_Phi} while minimizing the volume of $\bar{\mathcal{D}}(\Phi)$, i.e.,
\begin{equation}\label{exp:outerapx}
      \Phi_{\rm app}=\argmin_{\substack{\Phi\inX{S}{2n+m+T}\\\alpha_j\geq 0,j=1,\dots,J}}{\rm vol}(\bar{\mathcal{D}}(\Phi))\st\eqref{sum:lmi_Phi},
\end{equation}
where ${\rm vol}(\bar{\mathcal{D}}(\Phi))$ denotes the volume of $\bar{\mathcal{D}}(\Phi)$.
Then, the controller is designed using $\Phi_{\rm app}$ by
\begin{equation}\label{exp:feas_2step}
      \begin{split}
        \text{find }&P,L,\alpha,\beta,\\
        \st &P\succ 0,\alpha\geq 0,\beta>0,\text{\eqref{sum:lmi_qstab} with }\Phi=\Phi_{\rm app}.
      \end{split}
\end{equation}
Note that since the objective function of~\eqref{exp:outerapx} is non-convex due to the term $\Phi|\Phi_{22}$ in the computation of the volume, see~\cite[Sec. 2.2]{DDCQMI:BISOFFI2021}, we instead use a sub-optimal solution by following the approach in~\cite[Sec. 5.1]{DDCQMI:BISOFFI2021}.

\subsubsection{Visualized Example}
Similar to Sec.~\ref{subsec:1d_qstab}, we visualize the set of systems with the one-dimensional system $(A^*,B^*)=(1.2,0.6)$.
We collect data of length $T=20$ under element-wise bounded noise represented as 
\begin{equation}\label{exp:noiseset_element}
 \textstyle{
  \mathcal{D}_{\rm str}=\left\{
  \sum_{i=1}^{2n+m}\sum_{j=1}^T \univec{i}{2n+m}\delta_{ij}\univec{j}{T}^\top\middle|\delta_{ij}^2\leq \epsilon^2\right\},
  }
\end{equation}
where $\epsilon=0.15$.
We apply the co-design method \eqref{sum:LMI_dinfo} and the two-step method above to the same dataset.
As a result, \eqref{sum:LMI_dinfo} is feasible while \eqref{exp:feas_2step} is infeasible for $\Phi_{\rm app}$.
Let $\Phi_{\rm feas}, L_{\rm feas}$, and $P_{\rm feas}$ be the resulting feasible solution to~\eqref{sum:LMI_dinfo} and $K_{\rm feas}=L_{\rm feas}P_{\rm feas}^{-1}$ denote the designed stabilizing controller.
Fig~\ref{fig:1d_example_sum} illustrates
$\Sigma(\Phi_{\rm feas})$, all stabilized systems by $K_{\rm feas}$, and $\Sigma(\Phi_{\rm app})$ depicted by the dotted ellipse, the dashed band-shaped region, and the solid ellipse, respectively.
We can confirm that $\Sigma(\Phi_{\rm feas})$ obtained by our co-design method is included in the stabilized region, whereas $\Sigma(\Phi_{\rm app})$ is not.
This example shows that the outer approximation with $\Phi_{\rm feas}$ is more suitable for stabilizing controller synthesis, even though the volume of $\Sigma(\Phi_{\rm feas})$ is larger than that of $\Sigma(\Phi_{\rm app})$.

\begin{figure}[t]
  \centerline{\includegraphics[width=\columnwidth]{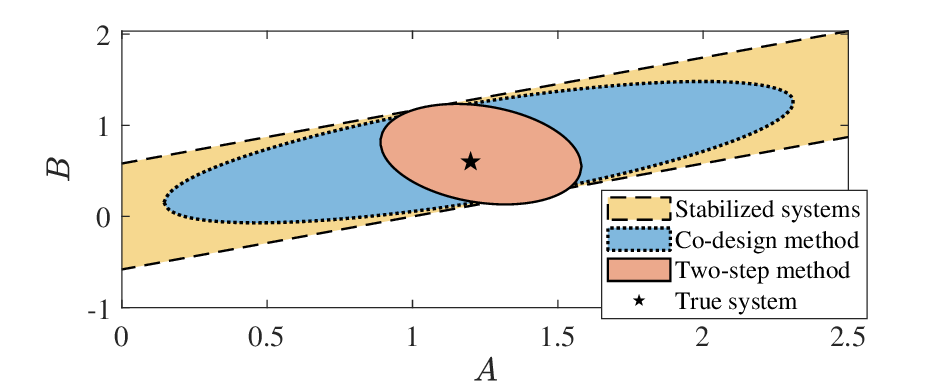}}
  \caption{Comparison of the sets of systems.}
  \label{fig:1d_example_sum}
\end{figure}

\subsubsection{Comparison of Feasibility}
We compare the feasibility rates of the two methods.
We use the transport aircraft model in Sec.~\ref{subsec:muri}, collect data of length $T=20$, and set the perturbation set as in \eqref{exp:noiseset_element} with appropriate dimensions.
For each value of $\epsilon$ in the range $10^{-3.3} \leq\epsilon\leq 10^{-2.3}$, we generate independent 50 datasets.
For each dataset, we evaluate the feasibility of the stabilizing controller design problem under both methods.
The proportion of feasible instances is plotted in Fig.~\ref{fig:codesign_frate}.
The results reveal a substantial difference in feasibility rates, particularly within the interval $10^{-2.8}\leq\epsilon\leq 10^{-2.6}$, where our proposed co-design method demonstrates over 75\% higher feasibility.
This significant improvement highlights the robustness and advantage of the proposed co-design approach.

\begin{figure}
  \centerline{\includegraphics[width=\columnwidth]{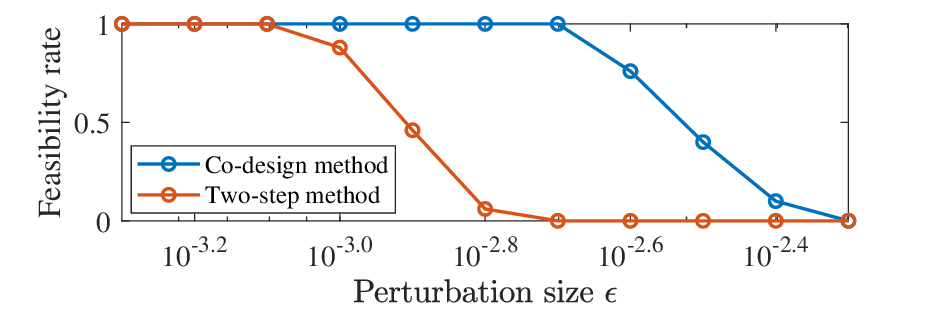}}
  \caption{Proportion of feasible datasets.}
  \label{fig:codesign_frate}
\end{figure}

\section{Conclusion}\label{sec:conclusion}
In this study, we have proposed a unified noise framework, referred to as data perturbation, that encompasses both exogenous disturbances and measurement noise, and analyzed data informativity under this framework.
We have first derived a sufficient LMI condition for data informativity with respect to quadratic stabilization and shown that this condition is also necessary under a structural condition on the perturbation.
Specifically, the set of systems consistent with the data has been characterized using a QMI, and the LMI condition has been derived through a newly proposed matrix S-procedure.
Furthermore, we have extended this analysis to optimal control and quadratic stabilization via output feedback.
We have also introduced a more general noise framework that can represent specific structures in data perturbation and provided a sufficient condition for data informativity under this framework.
This generalized noise model is approximated by the solution set of a single QMI, allowing it to be linked to the aforementioned data informativity problem.
By solving the approximation and controller design problems simultaneously, less conservative control design has been achieved.
The numerical experiments have visualized and validated these theoretical results.

Future work involves two key directions.
The first involves the application of the data perturbation framework to other control mechanisms, such as data-driven model predictive control.
The second direction is to develop an adaptive data-driven controller synthesis method based on the structured perturbation approach.

\appendix

We introduce a proposition which states that an image of a QMI set $\Zqr(\Pi)$ under linear maps is also a solution set of another QMI.
\begin{aproposition}[{\cite[Theorem 3.4]{DDCQMI:Waarde2023_siam_qmi}}]\label{prop:lin_eq_qmi}
    Let $\Pi\in\mathbf{\Pi}_{q,r}$ and $W\inX{R}{q\times p}$.
    Define $\Pi_W\coloneqq \blkdiag (W^\top,I)\,\Pi\,\blkdiag(W,I)$,
    and $\mathcal{S}\coloneqq\{ZW|Z\in\Zqr(\Pi)\}$.
    Then, $\mathcal{S}\subseteq\Zab{p}{r}(\Pi_W)$ holds.
    Moreover, if $W$ has full column rank or $\Pi_{22}$ is nonsingular, $\mathcal{S}=\Zab{p}{r}(\Pi_W)$ holds.
\end{aproposition}

We introduce another version of matrix S-procedure.
The following proposition is used in the proof of Theorem~\ref{thm:approximation}.
\begin{aproposition}[{\cite[Lemma 2]{DDCQMI:BISOFFI2021}}]\label{prop:Slem_lossy}
  Let $\Omega, \Pi_1,\dots,\Pi_J\inX{S}{q+r}$ hold.
  Then $\bigcap_{1\leq j\leq J}\Zqr(\Pi_j)\subseteq\Zqr(\Omega)$ holds if there exists $\alpha_j\geq 0, j=1,\dots,J$ such that
    $\Omega-\sum_{j=1}^{J}\alpha_j \Pi_j\succeq 0$.
\end{aproposition}

\bibliographystyle{IEEEtran}
\bibliography{IEEEabrv,Control,Robust_DDC,DDC,mathematics}

\begin{thebibliography}{10}
\providecommand{\url}[1]{#1}
\csname url@rmstyle\endcsname
\providecommand{\newblock}{\relax}
\providecommand{\bibinfo}[2]{#2}
\providecommand\BIBentrySTDinterwordspacing{\spaceskip=0pt\relax}
\providecommand\BIBentryALTinterwordstretchfactor{4}
\providecommand\BIBentryALTinterwordspacing{\spaceskip=\fontdimen2\font plus
\BIBentryALTinterwordstretchfactor\fontdimen3\font minus \fontdimen4\font\relax}
\providecommand\BIBforeignlanguage[2]{{%
\expandafter\ifx\csname l@#1\endcsname\relax
\typeout{** WARNING: IEEEtran.bst: No hyphenation pattern has been}%
\typeout{** loaded for the language `#1'. Using the pattern for}%
\typeout{** the default language instead.}%
\else
\language=\csname l@#1\endcsname
\fi
#2}}

\bibitem{DDC:Hou2013}
Z.-S. Hou and Z.~Wang, ``From model-based control to data-driven control: Survey, classification and perspective,'' \emph{Information Sciences}, vol. 235, pp. 3--35, 2013.

\bibitem{DDC:Waarde2020_TAC_Dinfo}
H.~J. van Waarde, J.~Eising, H.~L. Trentelman, and M.~K. Camlibel, ``Data informativity: A new perspective on data-driven analysis and control,'' \emph{{IEEE} Trans. Automat. Contr.}, vol.~65, no.~11, pp. 4753--4768, 2020.

\bibitem{DDC:Waarde2023_Cont_sys_mag_informativity}
H.~J. Van~Waarde, J.~Eising, M.~K. Camlibel, and H.~L. Trentelman, ``The informativity approach: To data-driven analysis and control,'' \emph{{IEEE} Control Syst. Mag.}, vol.~43, no.~6, pp. 32--66, 2023.

\bibitem{DDCQMI:Waarde2022_TAC_origin}
H.~J. van Waarde, M.~K. Camlibel, and M.~Mesbahi, ``From noisy data to feedback controllers: Nonconservative design via a matrix {S}-lemma,'' \emph{{IEEE} Trans. Automat. Contr.}, vol.~67, no.~1, pp. 162--175, 2022.

\bibitem{DDCQMI:Waarde2023_siam_qmi}
H.~J. van Waarde, M.~K. Camlibel, J.~Eising, and H.~L. Trentelman, ``Quadratic matrix inequalities with applications to data-based control,'' \emph{SIAM Journal on Control and Optimization}, vol.~61, no.~4, pp. 2251--2281, 2023.

\bibitem{DDCQMI:Steentjes2022_Cont_Sys_Let_Covariance}
T.~R.~V. Steentjes, M.~Lazar, and P.~M.~J. Van~den Hof, ``On data-driven control: Informativity of noisy input-output data with cross-covariance bounds,'' \emph{IEEE Control Systems Letters}, vol.~6, pp. 2192--2197, 2022.

\bibitem{DDCQMI:Waarde2024_TAC_AR}
H.~J. van Waarde, J.~Eising, M.~K. Camlibel, and H.~L. Trentelman, ``A behavioral approach to data-driven control with noisy input-output data,'' \emph{{IEEE} Trans. Automat. Contr.}, vol.~69, no.~2, pp. 813--827, 2024.

\bibitem{DDCQMI:Waarde2024ECC_disp}
E.~T. Nguyen and H.~J. Van~Waarde, ``Synthesis of dissipative systems using input-state data,'' in \emph{2024 European Control Conference (ECC)}, 2024, pp. 2959--2964.

\bibitem{DDCQMI:Burohman2023_TAC_reducedorder}
A.~M. Burohman, B.~Besselink, J.~M.~A. Scherpen, and M.~K. Camlibel, ``From data to reduced-order models via generalized balanced truncation,'' \emph{{IEEE} Trans. Automat. Contr.}, vol.~68, no.~10, pp. 6160--6175, 2023.

\bibitem{DDCQMI:BISOFFI2024_CSL}
A.~Bisoffi, L.~Li, C.~D. Persis, and N.~Monshizadeh, ``Controller synthesis for input-state data with measurement errors,'' \emph{IEEE Control Systems Letters}, vol.~8, pp. 1571--1576, 2024.

\bibitem{DDCQMI:Lidong2026_automatica}
L.~Li, A.~Bisoffi, C.~{De Persis}, and N.~Monshizadeh, ``Controller synthesis from noisy-input noisy-output data,'' \emph{Automatica}, vol. 183, p. 112545, 2026.

\bibitem{DDCQMI:Hu2025_RDPC}
K.~Hu and T.~Liu, ``Robust data-driven predictive control for unknown linear systems with bounded disturbances,'' \emph{{IEEE} Trans. Automat. Contr.}, pp. 1--16, 2025.

\bibitem{DDCQMI:BISOFFI2021}
A.~Bisoffi, C.~{De Persis}, and P.~Tesi, ``Trade-offs in learning controllers from noisy data,'' \emph{Systems \& Control Letters}, vol. 154, no. 104985, 2021.

\bibitem{DDCQMI:Hu2022_CDC}
K.~Hu and T.~Liu, ``Data-driven {H}$\infty$ control for unknown linear time-invariant systems with bounded disturbances,'' in \emph{2022 IEEE 61st Conference on Decision and Control (CDC)}, 2022, pp. 1423--1428.

\bibitem{DDCQMI:BISOFFI2022_Petersen}
A.~Bisoffi, C.~{De Persis}, and P.~Tesi, ``Data-driven control via {P}etersen's lemma,'' \emph{Automatica}, vol. 145, no. 110537, 2022.

\bibitem{DDCQMI:Kaminaga2025_ACC}
T.~Kaminaga and H.~Sasahara, ``Data informativity for quadratic stabilization under data perturbation,'' in \emph{2025 American Control Conference (ACC)}, 2025.

\bibitem{boyd1994linear}
S.~Boyd, L.~El~Ghaoui, E.~Feron, and V.~Balakrishnan, \emph{Linear Matrix Inequalities in System and Control Theory}.\hskip 1em plus 0.5em minus 0.4em\relax SIAM, 1994.

\bibitem{MATH:Polik2007}
I.~P\'{o}lik and T.~Terlaky, ``A survey of the {S}-lemma,'' \emph{SIAM Review}, vol.~49, no.~3, pp. 371--418, 2007.

\bibitem{Con:PETERSEN1987}
I.~R. Petersen, ``A stabilization algorithm for a class of uncertain linear systems,'' \emph{Systems \& Control Letters}, vol.~8, no.~4, pp. 351--357, 1987.

\bibitem{Math:Drusvyatskiy_FacialReduction}
D.~Drusvyatskiy and H.~Wolkowicz, ``The many faces of degeneracy in conic optimization,'' \emph{Foundations and Trends in Optimization}, vol.~3, no.~2, pp. 77--170, 2017.

\bibitem{Cntr:scherer2000linear}
C.~Scherer and S.~Weiland, ``Linear matrix inequalities in control,'' \emph{Lecture Notes, Dutch Institute for Systems and Control, Delft, The Netherlands}, vol.~3, no.~2, 2000.

\bibitem{Con:shimomura2005multiobjective}
T.~Shimomura and T.~Fujii, ``Multiobjective control via successive over-bounding of quadratic terms,'' \emph{International Journal of Robust and Nonlinear Control}, vol.~15, no.~8, pp. 363--381, 2005.

\bibitem{Con:Tran2012}
Q.~Tran~Dinh, S.~Gumussoy, W.~Michiels, and M.~Diehl, ``Combining convex–concave decompositions and linearization approaches for solving {BMIs}, with application to static output feedback,'' \emph{{IEEE} Trans. Automat. Contr.}, vol.~57, no.~6, pp. 1377--1390, 2012.

\bibitem{robert2004monte}
C.~P. Robert and G.~Casella, \emph{Monte Carlo Statistical Methods}, 2nd~ed.\hskip 1em plus 0.5em minus 0.4em\relax Springer New York, NY, 2004.

\bibitem{cntr:stevens2015aircraft}
B.~L. Stevens, F.~L. Lewis, and E.~N. Johnson, \emph{Aircraft Control and Simulation: Dynamics, Controls Design, and Autonomous Systems}.\hskip 1em plus 0.5em minus 0.4em\relax John Wiley \& Sons, 2015.

\end{thebibliography}

\begin{IEEEbiographynophoto}{Taira Kaminaga}
received the B.E. degree in engineering from Institute of Technology, Tokyo, Japan in 2023, and M.E. degree in engineering from Institute of Science Tokyo, Tokyo, Japan in 2025. He is currently pursuing the Ph.D. degree at the Graduate School of Information Science and Technology, the University of Tokyo, Tokyo, Japan.
His research interests include data-driven control.
\end{IEEEbiographynophoto}

\begin{IEEEbiographynophoto}{Hampei Sasahara}
is Lecturer with the Department of Information Physics and Computing, Graduate School of Information Science and Technology, the University of Tokyo, Tokyo, Japan. He received the Ph.D. degree in engineering from Tokyo Institute of Technology in 2019. From 2019 to 2021, he was a Postdoctoral Scholar with KTH Royal Institute of Technology, Stockholm, Sweden. From 2022 to 2024, he was an Assistant Professor with Tokyo Institute of Technology, Tokyo, Japan. From 2024 to 2025, he was an Assistant Professor with Institute of Science Tokyo, Tokyo, Japan. His main interests include data-driven control and secure control system design. 
\end{IEEEbiographynophoto}

\end{document}